\documentclass[fontsize=11pt,a4paper,DIV12]{scrartcl}

\usepackage[boldsans]{concmath} %
\usepackage{euler} %
\usepackage{euscript}

\usepackage[utf8]{inputenc}
\usepackage{amsthm,amsmath,amssymb}
\usepackage{enumerate,enumitem}
\usepackage{graphicx}
\graphicspath{{Figures/}}
\input{pdffig.sty}
\def\term#1{\emph{#1}\marginpar{\raggedright{\small\em #1}}}

\usepackage{color}

\newenvironment{claimproof}{%
 \noindent%
 \textit{Proof of Claim.}%
}
{
\hfill $\triangle$%
\medskip
}

\usepackage{thmtools}
\usepackage[hidelinks]{hyperref}
\usepackage[capitalize]{cleveref}
\usepackage{cite}

\theoremstyle{plain}
\newtheorem{thm}{Theorem}

\crefname{theorem}{Theorem}{Theorems}
\newtheorem{corollary}[thm]{Corollary}
\crefname{corollary}{Corollary}{Corollaries}
\newtheorem{proposition}[thm]{Proposition}
\crefname{proposition}{Proposition}{Propositions}
\newtheorem{lemma}[thm]{Lemma}
\crefname{lemma}{Lemma}{Lemmas}
\newtheorem{myobservation}{Observation}
\crefname{myobservation}{Observation}{Observations}

\newtheorem{remark}{Remark}
\crefname{remark}{Remark}{Remarks}

\crefname{conjecture}{Conjecture}{Conjectures}
\newtheorem*{quest*}{Question}
\newtheorem{myclaim}{Claim}
\crefname{myclaim}{Claim}{Claims}
\crefname{figure}{Figure}{Figures}

\usepackage{thm-restate} 
\usepackage{etoolbox}

\def\R{\mathcal{R}}
\def\RR{\mathbb{R}}
\def\VV{\mathcal{V}}
\def\VVup{\mathcal{V}^\uparrow}
\def\SS{\mathfrak{S}}
\def\CC{\mathcal{C}}
\def\ovl{\overline}
\def\BB{\mathcal{B}}

\def\partialW{\partial{}W}

\def\SC{\scshape}

\title{Plattenbauten: Touching Rectangles in Space%%
  \thanks{An extended abstract appears in the proceedings of the 46th International
    Workshop on Graph-Theoretic Concepts in Computer Science (WG 2020)~\cite{WG-Plattenbauten}.}}
\author{
\parbox{7.5cm}{\center
{\SC Stefan Felsner%
\thanks{Partially supported by DFG grant FE-340/11-1.}
}\\[3pt]
\normalsize
\small
  {Institut f\"ur Mathematik\\
   Technische Universit\"at Berlin}
}
\parbox{7.5cm}{\center
{\SC Kolja Knauer%
}\\[3pt]
\normalsize
\small
{Departament de Matem\`atiques i Inform\`atica,\\ Universitat de Barcelona\\
  and\\
   Aix-Marseille Univ, Universit\'e de Toulon,\\
   CNRS, LIS, Marseille}
}\\
\parbox{7.5cm}{\center
{\SC Torsten Ueckerdt%
}\\[3pt]
\normalsize
\small
  {Institute for Theoretical Informatics,\\
   Karlsruhe Institute of Technology (KIT)}
}
}%%END AUTHOR

\date{}%% 3. Sept.2020
\begin{document}

\maketitle

\begin{abstract}
  Planar bipartite graphs can be represented as touching graphs of horizontal
  and vertical segments in~$\RR^2$. We study a generalization in space:
  touching graphs of axis-aligned rectangles in~$\RR^3$, and prove that planar
  3-colorable graphs can be represented this way.  The result implies a
  characterization of corner polytopes previously obtained by Eppstein and
  Mumford. A by-product of our proof is a distributive lattice structure on
  the set of orthogonal surfaces with given skeleton.
    
  Further, we study representations by axis-aligned non-coplanar rectangles
  in~$\RR^3$ such that all regions are boxes. We show that the resulting
  graphs correspond to octahedrations of an octahedron. This generalizes the
  correspondence between planar quadrangulations and families of horizontal
  and vertical segments in~$\RR^2$ with the property that all regions are
  rectangles.
  \medskip

  \noindent{{\bfseries Keywords:} Touching graphs, contact graphs, boxicity, planar graphs}
\end{abstract}

%%%%%%%%%%%%%%%%%%%%%%%%%%%%%
%%                         %%
%%        SECTION 1        %%
%%                         %%
%%%%%%%%%%%%%%%%%%%%%%%%%%%%%
\section{Introduction}

The importance of contact and intersection representations of graphs stems not only from their numerous applications including information visualization, chip design, bio informatics and robot motion planning (see for example the references in~\cite{buchsbaum2008rectangular,felsner2013rectangle}), but also from the structural and algorithmic insights accompanying the investigation of these intriguing geometric arrangements.
From a structural point of view, the certainly most fruitful contact representations (besides the ``Kissing Coins'' of Koebe, Andrew, and Thurston~\cite{koebe1936kontaktprobleme,bowers2010circle,stephenson2005introduction}) are axis-aligned segment contact representations:
families of interior-disjoint horizontal and vertical segments in~$\RR^2$ where the intersection of any two segments is either empty or an endpoint of at least one of the segments.
The corresponding touching graph\footnote{We use the term touching graphs rather than the more standard contact graph to underline the fact that segments with coinciding endpoints (e.g., two horizontal segments touching a vertical segment in the same point but from different sides, but also non-parallel segments with coinciding endpoint) do not form an edge.} has the segments as its vertices and the pairs of segments as its edges for which an endpoint of one segment is an interior point of the other segment, see the left of~\cref{fig:example}.
It has been discovered several times~\cite{hartman1991grid,pach1994representation} that any such touching graph is bipartite and planar, and that these two obviously necessary conditions are in fact already sufficient: Every planar bipartite graph is the touching graph of interior-disjoint axis-aligned segments in~$\RR^2$.
In fact, edge-maximal segment contact representations endow their associated plane graphs with many useful combinatorial structures such as 2-orientations~\cite{felsner2013rectangle}, separating decompositions~\cite{de2001topological}, bipolar orientations~\cite{rosenstiehl1986rectilinear,tamassia1986unified}, transversal structures~\cite{fusy2007combinatoire}, and Schnyder woods~\cite{ueckerdt-phd}.

\begin{figure}
 \centering
 \includegraphics[width=0.9\textwidth]{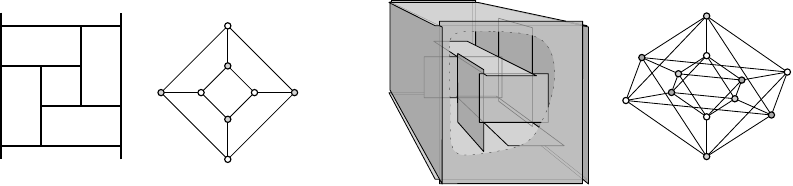}
 \caption{An axis-aligned segment contact representation (left) and a Plattenbau (right) together with the respective touching graphs.}
 \label{fig:example}
\end{figure}

In this paper we extend axis-aligned segment contact representations in~$\RR^2$ to axis-aligned rectangle contact representations in~$\RR^3$.
That is, we consider families~$\R$ of axis-aligned closed and bounded rectangles in~$\RR^3$ with the property that for all~$R,R'\in \R$ the intersection~$R\cap R'$ is a subset of the boundary of at least one of them, i.e., the rectangles are interiorly disjoint.
We call such a family a \term{Plattenbau}\footnote{Plattenbau (plural Plattenbauten) is a German word describing a building (\emph{Bau}) made of prefabricated concrete panels (\emph{Platte}).}. Given a Plattenbau~$\R$ one can consider its \emph{intersection graph}~$I_{\R}$, see Section~\ref{sec:Conclusions}. However, for us the more important concept is a certain subgraph of~$I_{\R}$, called the \term{touching graph}~$G_{\R}$ of~$\R$. There is one vertex in~$G_{\R}$ for each rectangle in~$\R$ and two vertices are adjacent if the corresponding rectangles \term{touch}, i.e., their intersection is non-empty and contains interior points of one and only one of the rectangles.
We say that~$G$ is a \term{Plattenbau graph} if there is a Plattenbau~$\R$ such that~$G\cong G_{\R}$. In this case we call~$\R$ a \term{Plattenbau representation} of~$G$.

Plattenbauten are a natural generalization of axis-aligned segment contact representations in~$\RR^2$ and thus Plattenbau graphs are a natural generalization of planar bipartite graphs.
While clearly all Plattenbau graphs are tripartite (properly vertex 3-colorable), it is an interesting challenge to determine the exact topological properties in~$\RR^3$ that hold for all Plattenbau graphs, thus generalizing the concept of planarity from $2$ to $3$ dimensions (for tripartite graphs).
We present results towards a characterization of Plattenbau graphs in three directions.

\paragraph*{Our Results and Organization of the Paper.}

In~\cref{sec:types-and-augmentation} we provide examples of Plattenbau graphs
and give some necessary conditions for all Plattenbau graphs.  We
observe that unlike touching graphs of segments, general Plattenbau graphs are
not closed under taking subgraphs.  We circumvent this issue by restricting
ourselves to {generic} Plattenbauten, i.e., $\R$ contains no coplanar
rectangles.
Moreover, we introduce boxed Plattenbauten where every bounded region of~$\RR^3$ is a box, and discuss questions of augmentability.

In~\cref{sec:planar} we show that within planar graphs the necessary condition of 3-colorability is also sufficient for Plattenbau graphs.
Thus, the topological characterization of Plattenbau graphs must fully contain planarity (which is not obvious as we consider 3-colorable graphs and not only bipartite graphs).
  
\begin{restatable}{theorem}{PLANAR}\label{thm:planar}
 Every 3-colorable planar graph is the touching graph of a generic Plattenbau.
\end{restatable}

Along the proof of \cref{thm:planar}, we obtain a characterization of skeletons of
orthogonal surfaces which is implicit already in work of Eppstein and
Mumford~\cite{EppM-14}. Another proof of \cref{thm:planar} can be obtained from
Gon\c{c}alves' proof that~3-colorable planar graphs admit segment
intersection representations with segments of 3
slopes~\cite{g-3cpghaicru3s-19}. We further comment on these alternative
approaches in \cref{ssec:EM+G}.  A
consequence of our approach is a natural partial order - namely a distributive lattice - on the set of
orthogonal surfaces with a given skeleton.

In~\cref{sec:proper-boxed} we consider generic and boxed Plattenbau graphs as the 3-dimensional analogue of edge-maximal planar bipartite graphs, the quadrangulations.
We give a complete characterization:

\begin{restatable}{theorem}{OCTA}\label{thm:proper-boxed}
 A graph~$G$ is the touching graph of a generic boxed Plattenbau~$\R$ if and only if there are six outer vertices in~$G$ such that each of the following holds:
 \begin{enumerate}[leftmargin=2.3em,label = (P\arabic*)]
  \item $G$ is connected and the outer vertices of~$G$ induce an octahedron.\label{P1:outer-octahedron}
  \item The edges of~$G$ admit an orientation such that\label{P2:4-orientation}
   \begin{itemize}
    \item the bidirected edges are exactly the outer edges,
    \item each vertex has exactly 4 outgoing edges.
   \end{itemize}
  \item The neighborhood~$N(v)$ of each vertex~$v$ induces a spherical quadrangulation~$SQ(v)$ in which the out-neighbors of~$v$ induce a 4-cycle.\label{P3:quadrangulation}
   \begin{itemize}
    \item If~$v$ is an outer vertex, this 4-cycle bounds a face of~$SQ(v)$.
   \end{itemize}
  \item For every edge~$uv$ of~$G$ with common neighborhood~$C = N(u) \cap N(v)$, the cyclic ordering of~$C$ around~$u$ in~$SQ(v)$ is the reverse of the cyclic ordering of~$C$ around~$v$ in~$SQ(u)$.\label{P4:common-neighbors}
 \end{enumerate}
\end{restatable}

A \term{spherical quadrangulation} is a graph embedded on the 2-dimensional sphere without crossings such that each face is bounded by a 4-cycle.
Spherical quadrangulations are 2-connected, planar, and bipartite.
We remark that \cref{thm:proper-boxed} does not give a complete characterization of generic Plattenbau graphs since some generic Plattenbau graphs are not contained in any generic boxed Plattenbau graph as discussed in~\cref{sec:types-and-augmentation}.

Let us further remark that we show in Subsection~\ref{subsec:iterative} how every generic boxed Plattenbau can be constructed in a natural way from
trivial parts.

%%%%%%%%%%%%%%%%%%%%%%%%%%%%%
%%                         %%
%%        SECTION 2        %%
%%                         %%
%%%%%%%%%%%%%%%%%%%%%%%%%%%%%
\section{Types of Plattenbauten and Questions of Augmentation}
\label{sec:types-and-augmentation}

Let us observe some properties of Plattenbau graphs.
Clearly, the class of all Plattenbau graphs is closed under taking induced subgraphs.
Examples of Plattenbau graphs are~$K_{2,2,n}$, see~\cref{fig:K225}, and the class of grid intersection graphs, i.e., bipartite intersection graphs of axis-aligned segments in the plane~\cite{hartman1991grid}.
For the latter take the segment intersection
representation of a graph, embed it into the $xy$-plane in~$\RR^3$ and thicken all horizontal segments a small amount into $y$-direction and all vertical segments a bit into $z$-direction outwards the $xy$-plane.
In particular,~$K_{m,n}$ is a Plattenbau graph, see~\cref{fig:K225}.
In order to exclude some graphs, we observe some necessary properties of all Plattenbau graphs.

%%%%%%%%%%%%%%%%%%%%%%%%%%%%%%%%%%%%%%%%%%%%%%%%%%%%%%%%
 \begin{figure}
  \centering
  \hfill
  \includegraphics{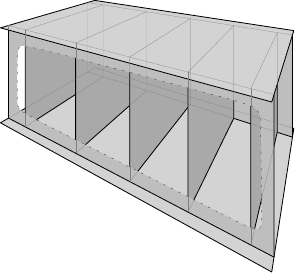}
  \hfill
  \includegraphics{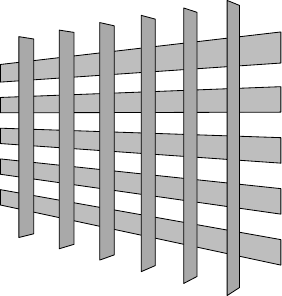}
  \hfill
  \ 
  \caption{Plattenbau representations of~$K_{2,2,5}$ (left) and~$K_{5,6}$ (right).}
  \label{fig:K225}
 \end{figure}
%%%%%%%%%%%%%%%%%%%%%%%%%%%%%%%%%%%%%%%%%%%%%%%%%%%%%%%%

\begin{myobservation}\label{obs:easy}
 If~$G$ is a Plattenbau graph, then
 \begin{enumerate}
  \item the chromatic number of~$G$ is at most 3,\label{enum:3chromatic}
  \item the neighborhood of any vertex of~$G$ is planar,\label{enum:planar}
  \item the \emph{boxicity} of~$G$, i.e., the smallest dimension~$d$ such that~$G$ is the intersection graph of some boxes in~$\RR^d$, is at most 3.\label{enum:boxicity}
 \end{enumerate}
\end{myobservation}
\begin{proof}  
 \cref{enum:3chromatic}:
 Each orientation class is an independent set.
   
 \cref{enum:planar}: 
 Let~$v$ be a vertex of~$G$ represented by~$R\in\R$.
 Let~$H$ be the supporting hyperplane of~$R$ and~$H^+, H^-$ the corresponding open halfspaces.  
 The neighborhood~$N(v)$ consists of rectangles~$R^+$ intersecting~$H^+$ and those~$R^-$ intersecting~$H^-$.
 The rectangles in each of these sets have a plane touching graph, since it corresponds to the touching graph of the axis-aligned segments given by their intersections with~$R$. 
 The neighboring rectangles in~$R^+\cap R^-$ are on the outer face in both graphs in opposite order, so identifying them gives a planar drawing of the graph induced by~$N(v)$.
 See \cref{fig:neighborhood} for an illustration.

 %%%%%%%%%%%%%%%%%%%%%%%%%%%%%%%%%%%%%%%%%%%%%%%%%%%%%%%%
 \begin{figure}
  \centering
  \includegraphics{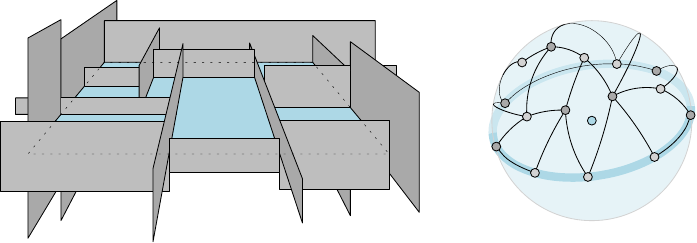}
   \caption{
    Left: A $z$-rectangle (depicted in blue) in a Plattenbau with its touching rectangles intersecting its upper halfspace~$H^+$.
    Right: The resulting crossing-free embedding on the upper hemisphere.}
  \label{fig:neighborhood}
 \end{figure}
 %%%%%%%%%%%%%%%%%%%%%%%%%%%%%%%%%%%%%%%%%%%%%%%%%%%%%%%%
 
 \cref{enum:boxicity}:
 A Plattenbau~$\R$ can be transformed into a set~$\mathcal{B}$ of boxes such that the touching graph of~$\R$ is the intersection graph of~$\mathcal{B}$ as follows:
 First, shrink each rectangle orthogonal to the $i$-axis by a small enough~$\varepsilon > 0$ in both dimensions different from~$i$.
 As a result, we obtain a set of pairwise disjoint rectangles.
 Then, expand each such rectangle by~$\varepsilon$ in dimension~$i$.
 The obtained set~$\mathcal{B}$ of boxes are again interiorly disjoint and all intersections are touchings.
\end{proof}

Note that for \cref{enum:3chromatic,enum:planar} of \cref{obs:easy} it is
crucial that~$G$ is the touching graph and not the intersection graph.
Moreover, \cref{obs:easy} allows to reject some graphs as Plattenbau graphs:
\begin{itemize}
 \item $K_4$ is not a Plattenbau graph (by \cref{enum:3chromatic} of \cref{obs:easy}), 
 \item $K_{1,3,3}$ is not a Plattenbau graph (by \cref{enum:planar} of \cref{obs:easy}).
 \item The full subdivision of~$K_{2^{2^5}+1}$ is not a Plattenbau graph (by \cref{enum:boxicity} of \cref{obs:easy} and~\cite{Cha-11}).
\end{itemize}

\begin{remark}
  The number of vertices of the last example is a 10 digits number. This can
  be reduced: A class $\cal S$ of geometric objects in $\RR^d$ is
  \textit{$t$-separable} if there exists a family
  ${\cal H} = \{H_1,\ldots, H_t\}$ of hyperplanes, such that any two disjoint
  elements of $\cal S$ can be separated by a translate of one of the
  hyperplanes from $\cal H$.  A family $\cal B$ of axis-aligned boxes in
  $\RR^3$ is clearly 3-separable.  In~\cite[Prop.~2.3]{f-odpmr-14} it has
  been shown that if $G$ is a bipartite graph admitting a $t$-separable
  intersection representation, then the bipartite poset corresponding to $G$
  has order dimension at most $2t$.  Since it is known that the order
  dimension of the full subdivision of $K_{2647}$ is 7 (Ho\c{s}ten and
  Morris~\cite{hm-odcg-99}) we conclude:

 \begin{itemize}
  \item the full subdivision of $K_{2647}$ is not a Plattenbau graph.
 \end{itemize}
\end{remark}

In particular, some bipartite graphs are not Plattenbau graphs.  Together
with~$K_{m,n}$ being a Plattenbau graph, this shows that the class of
Plattenbau graphs is not closed under taking subgraphs; an unusual situation
for touching graphs, which prevents us from solely focusing on edge-maximal
Plattenbau graphs.  To overcome this issue, we say that a Plattenbau~$\R$ is
\term{generic} if it contains no co-planar rectangles\footnote{In the
  conference version of this paper, we worked with a more restrictive version
  of ``proper'' Plattenbau, where for any two
touching rectangles $R$, $R'$ the intersection $R\cap R'$ must be a boundary edge of
one of $R$ and~$R'$. However, there was a mistake in one proof whence we
cannot ensure such a representations for planar 3-chromatic graphs.}.
In a generic Plattenbau each edge of each rectangle intersects the interior of at
most one other rectangle. Thus, if~$\R$ is generic, then each edge of the
touching graph~$G_\R$ can be removed by shortening one of the participating
rectangles slightly.  That is, the class of graphs with generic Plattenbau
representations is closed under subgraphs.

We furthermore say that a Plattenbau~$\R$ is \term{boxed} if six \emph{outer rectangles} constitute the sides of a box that contains all other rectangles and all regions inside this box are also boxes.
(A box is an axis-aligned full-dimensional cuboid, i.e., the Cartesian product of three bounded intervals of non-zero length.
And a region is a connected component of $\mathbb{R}^3$ after the removal of all rectangles in $\R$.)
For boxed Plattenbauten we use the additional convention that the edge-to-edge intersections of outer rectangles yield edges in the touching graph, even though these intersections contain no interior points.
In particular, the outer rectangles of a generic boxed Plattenbau induce an octahedron in the touching graph.

\begin{myobservation}\label{obs:proper-gives-sparse}
 The touching graph~$G_\R$ of a generic Plattenbau~$\R$ with~$n \geq 6$ vertices has at most~$4n-12$ edges.
 Equality holds if 
%  and only if
 $\R$ is boxed.
\end{myobservation}
\begin{proof}
 As noted above, for a generic Plattenbau~$\R$ with touching graph~$G_\R$ there is an injection from the edges of~$G_\R$ to the edges of rectangles in~$\R$:
 For each edge~$uv$ in~$G_\R$ with corresponding rectangles~$R_u,R_v \in \R$, take the edge of~$R_u$ or~$R_v$ that contributes to their intersection~$R_u \cap R_v$.
 This way, each of the four edges of each of the~$n$ rectangles in~$\R$ corresponds to at most one edge in~$G_\R$.
 
 Moreover, if~$\R$ contains at least two rectangles of each orientation, the bounding box of~$\R$ contains at least 12 edges of rectangles in its boundary, none of which corresponds to an edge in~$G_\R$.
 Thus, in this case~$G_\R$ has at most~$4n-12$ edges.
 Otherwise, for one of the three orientations,~$\R$ contains at most one rectangle in that orientation.
 In this case,~$G_\R$ is a planar bipartite graph plus possibly one additional vertex.
 In particular,~$G_\R$ has at most~$2(n-1) - 4 + (n-1) < 4n-12$ edges, as long as~$n \geq 6$.
 
 Finally, if $\R$ is boxed, then the above analysis is tight, i.e., $G_\R$ has exactly $4n-12$ edges in this case.
%  
 %% old version below %%
%  Finally, in order to have exactly~$4n-12$ edges, the above analysis must be tight.
%  This implies that~$\R$ has at least two rectangles of each orientation and its bounding box contains exactly $12$ edges of rectangles, i.e.,~$\R$ is boxed.
\end{proof}

An immediate consequence of \cref{obs:proper-gives-sparse} is that~$K_{5,6}$
is a Plattenbau graph which has no generic Plattenbau representation.
Contrary to the case of axis-aligned segments in~$\RR^2$, neither can every
generic Plattenbau in~$\RR^3$ be completed to a boxed Plattenbau, nor is every
boxed Plattenbau equivalent to a generic one. See \cref{fig:incompletable} for
problematic examples. The example on the left is generic, but it is not a
subgraph of a Plattenbau graph with a generic and boxed Plattenbau
representation. The touching graph of the example on the right is 7-regular
and has 12 vertices, i.e., 42 edges. Hence, by \cref{obs:proper-gives-sparse}
it has too many edges to be the touching graph of a generic Plattenbau.

%%%%%%%%%%%%%%%%%%%%%%%%%%%%%%%%%%%%%%%%%%%%%%%%%%%%%%%%
 \begin{figure}
  \centering
  \includegraphics{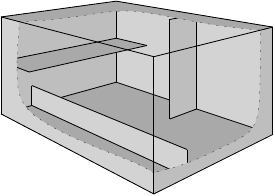}
  \caption{
   A generic Plattenbau that cannot be augmented to a boxed Plattenbau and a boxed Plattenbau that is not generic.
  }
  \label{fig:incompletable}
 \end{figure}
%%%%%%%%%%%%%%%%%%%%%%%%%%%%%%%%%%%%%%%%%%%%%%%%%%%%%%%%

%%%%%%%%%%%%%%%%%%%%%%%%%%%%%
%%                         %%
%%        SECTION 3        %%
%%                         %%
%%%%%%%%%%%%%%%%%%%%%%%%%%%%%
\section{Planar 3-Colorable Graphs}\label{sec:planar}

Let us recall the main result of this section:

\PLANAR*

The proof of this theorem is in several steps. First we introduce
orthogonal surfaces and show that the dual graph of the skeleton
of an orthogonal surface is a Plattenbau graph (\cref{prop:OS-RC}).  In
the second step we characterize triangulations whose dual is the
skeleton of an orthogonal surface (\cref{prop:babet}). One consequence
of this is a natural very well-behaved partial order, namely a distributive lattice, on the set
of orthogonal surfaces with given skeleton (\cref{cor:dist}).
We then show that a Plattenbau representation of a 3-colorable
triangulation can be obtained by patching orthogonal surfaces
in corners of orthogonal surfaces (\cref{sec:patching}). 

We begin with an easy observation.
\begin{myobservation} 
  Every 3-colorable planar graph~$G$ is an induced subgraph of a
  3-colorable planar triangulation.
\end{myobservation}
\begin{proof}[Sketch]
  Consider~$G$ with a plane embedding.  By adding just subdivided
  edges we find a 2-connected 3-colorable~$G'$ which has~$G$ as an induced
  subgraph.
 
  Fix a 3-coloring of~$G'$.  Let~$f$ be a face of~$G'$ of size at least four
  and~$c$ be a color such that at least three vertices of~$f$ are not
  colored~$c$. Stack a vertex~$v$ inside~$f$ and connect it to the vertices
  on~$f$ that are not colored~$c$.  The new vertex~$v$ is colored~$c$ and the
  sizes of the new faces within~$f$ are 3 or~4.  After stacking in a
  4-face, the face is either triangulated or there is a color which is not
  used on any newly created 4-face. A second stack triangulates it.
\end{proof}

A plane triangulation~$T$ is 3-colorable if and only if it is Eulerian.
Hence, the dual graph~$T^*$ of~$T$ apart from being 3-connected, cubic, and planar is also bipartite.
The idea of the proof is to find an orthogonal surface~$\SS$ such that~$T^*$ is the skeleton of~$\SS$. 
This is not always possible but with a technique of patching one orthogonal surface in an appropriate corner of a Plattenbau representation obtained from another orthogonal surface, we shall get to a proof of the theorem.

Consider~$\RR^3$ with the \emph{dominance order}, i.e.,~$x\leq y$ if and only
if $x_i \leq y_i$ for~$i=1,2,3$. The join and meet of this distributive
lattice are the componentwise $\max$ and $\min$. Let $\VV\subseteq\RR^3$ be a
finite \emph{antichain}, i.e., a set of mutually incomparable points. The
\emph{filter} of~$\VV$ is the
set~$\VVup := \{x\in\RR^3\mid \exists{v\in\VV}:v \leq x\}$ and the
boundary~$\SS_{\VV}$ of~$\VVup$ is the \term{orthogonal surface} generated
by~$\VV$. The left part of \cref{fig:OS-example} shows an example
in~$\RR^3$. The nine vertices of the \emph{generating set} $\VV$ are
emphasized.

%%%%%%%%%%%%%%%%%%%%%%%%%%%%%%%%%%%%%%%%%%%%%%%%%%%%%%%%
 \begin{figure}
  \centering
  \includegraphics[width =.9\textwidth]{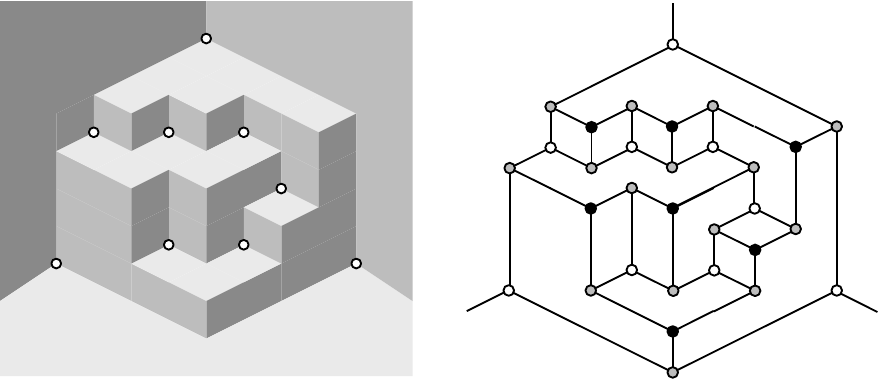}
  \caption{An orthogonal surface and its skeleton (vertex~$v_\infty$ omitted).}
  \label{fig:OS-example}
 \end{figure}
%%%%%%%%%%%%%%%%%%%%%%%%%%%%%%%%%%%%%%%%%%%%%%%%%%%%%%%%

Orthogonal surfaces have been studied by Scarf~\cite{s-cee-73} in the context of test sets for integer programs.
They later became of interest in commutative algebra, cf.\ the monograph of Miller and Sturmfels~\cite{ms-cca-04}.
Miller~\cite{m-pgmrtmi-02} observed the connections between orthogonal surfaces, Schnyder woods and the Brightwell-Trotter Theorem about the order dimension of polytopes, see also~\cite{f-gga-04}.

A maximal connected set of points of an orthogonal surface which is constant in
one of the coordinates is called a \emph{flat}. A non-empty
intersection of two flats is an \emph{edge}.
A point contained in three flats is called a \emph{vertex}. An edge incident to only one
vertex is a \emph{ray}.  We will only consider orthogonal surfaces
obeying the following non-degeneracy conditions:
(1)~Every vertex has degree exactly three.
(2)~There are exactly three rays.

The \term{skeleton}~$G_\SS$ of an orthogonal surface is the graph consisting of the
vertices and edges of the surface, in addition there is a 
vertex~$v_\infty$ which serves as second vertex of each ray.
The skeleton graph is planar, cubic, and bipartite.
The bipartition consists of the maxima and minima of the surface in
one class and of the saddle vertices in the other class.
The vertex~$v_\infty$ is a saddle vertex.
The dual of~$G_\SS$ is a triangulation with a designated outer face,
the dual of~$v_\infty$.

The generic structure of a bounded flat is as shown in
\cref{fig:gen-flat}; the boundary consists of two zig-zag paths
sharing the two \textit{extreme points} of the flat. The minima of the \emph{lower}
zig-zag are elements of the generating set~$\VV$, they are minimal
elements of the orthogonal surface~$\SS$. The maxima of the \emph{upper} zig-zag
are maximal elements of~$\SS$.  The maxima  can be considered
to be \textit{dual generators} of~$\SS$. 

%%%%%%%%%%%%%%%%%%%%%%%%%%%%%%%%%%%%%%%%%%%%%%%%%%%%%%%%
 \begin{figure}
  \centering
  \includegraphics[width =.65\textwidth]{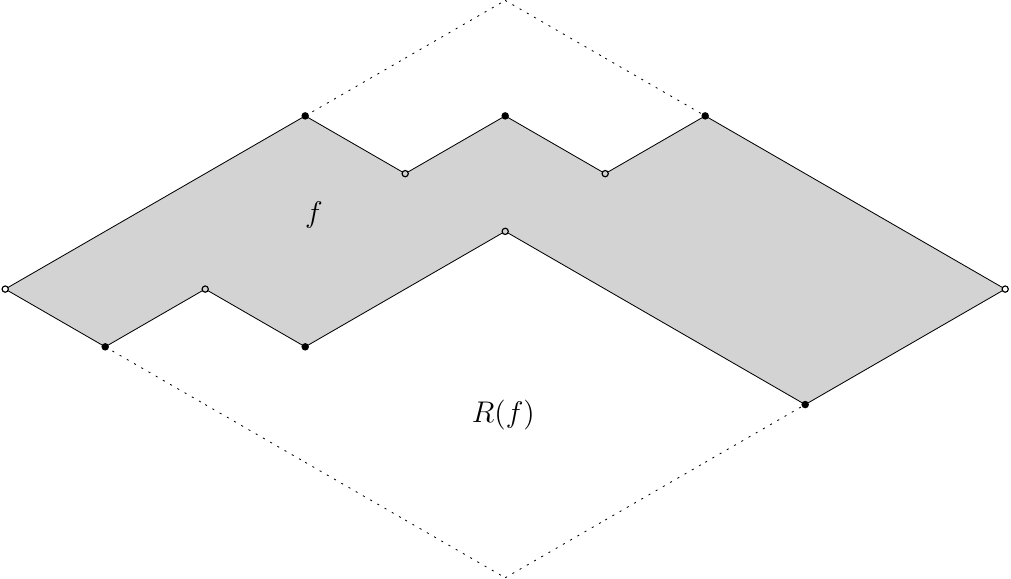}
  \caption{Generic flat~$f$ and the spanned rectangle~$R(f)$.}
  \label{fig:gen-flat}
 \end{figure}
%%%%%%%%%%%%%%%%%%%%%%%%%%%%%%%%%%%%%%%%%%%%%%%%%%%%%%%%

With the following proposition we establish a first connection between
orthogonal surfaces and Plattenbau graphs.

\begin{proposition}\label{prop:OS-RC}
  The dual triangulation of the skeleton of an orthogonal surface~$\SS$
  obeying the two non-degeneracy conditions is a Plattenbau graph and admits a
  generic Plattenbau representation.
\end{proposition}

\begin{proof}
  Choose a point not on~$\VV$ on each of the
  three rays of~$\SS$ and call these points the extreme points of
  their incident unbounded flats. 
  
  The two extreme points~$a_f,b_f$ of a flat~$f$ of~$\SS$ span a rectangle
  ~$R(f)$. Note that the other two corners of~$R(f)$ are~$\max(a_f,b_f)$
  and~$\min(a_f,b_f)$. We claim that the collection of rectangles~$R(f)$ is a
  \emph{weak}\marginpar{\raggedright{\small\em weak representation}} rectangle
  contact representation of the dual triangulation~$T$ of the skeleton
  of~$\SS$. Here weak means that the contacts of pairs of rectangles of
  different orientation can be an edge to edge contact. If~$f$ and~$f'$ share
  an edge~$e$ of the skeleton, then since one of the ends of~$e$ is a saddle
  point of~$\SS$ and thus extreme in two of its incident flats, it is extreme
  for at least one of~$f$ and~$f'$.  This shows that~$e$ is contained in the
  boundary of at least one of the rectangles~$R(f), R(f')$, i.e., the
  intersection of the open interiors of the rectangles is empty.

  Let~$f$ and~$f'$ be two flats. Let~$H_f$ and~$H_{f'}$ be
  the supporting planes.  If~$f$ is contained in an open halfspace~$O$ defined
  by~$H_{f'}$, then~$\max(a_f,b_f)$ and~$\min(a_f,b_f)$, the other two corners
  of~$R(f)$, are also in~$O$, hence~$R(f) \subset O$
  and~$R(f) \cap R(f') = \emptyset$.  If~$f$ intersects~$H_{f'}$ and~$f'$
  intersects~$H_{f}$, then consider the line~$\ell = H_{f'}\cap H_{f}$. This
  line is parallel to one of the axes, hence it intersects~$\SS$ in a closed
  interval~$I_\SS$.  If~$I_f$ and~$I_{f'}$ are the intervals obtained by
  intersecting~$\ell$ with~$f$ and~$f'$ respectively, then one of them
  equals~$I_\SS$ and the other is an edge of the skeleton of~$\SS$,
  i.e.,~$(f,f')$ is an edge of the triangulation~$T$.

 %%%%%%%%%%%%%%%%%%%%%%%%%%%%%%%%%%%%%%%%%%%%%%%%%%%%%%%%
 \begin{figure}
  \centering
  \includegraphics[width =.5\textwidth]{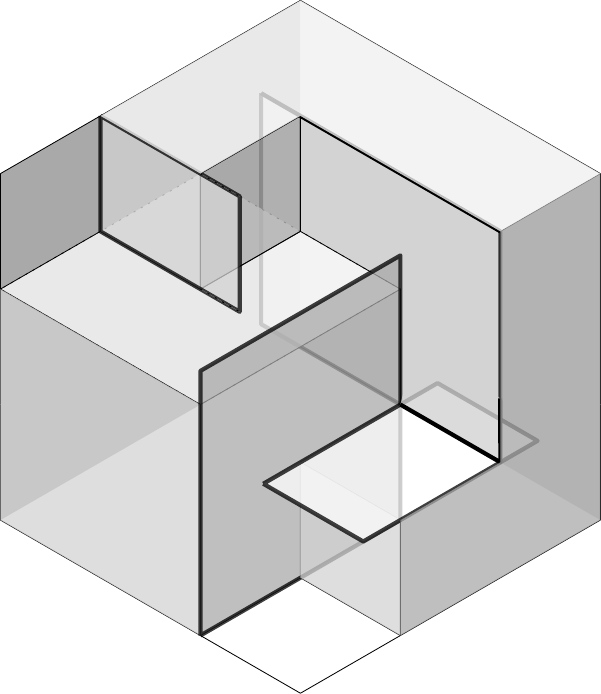}
  \caption{Replacing flats by rectangles and expanding in order to avoid weak contacts.}
  \label{fig:expand}
 \end{figure}
 %%%%%%%%%%%%%%%%%%%%%%%%%%%%%%%%%%%%%%%%%%%%%%%%%%%%%%%%
  
  It remains to expand some of the rectangles to change weak contacts
  into true contacts. Let~$e = f\cap f'$ be an edge such that the
  contact of~$R(f)$ and~$R(f')$ is weak. Select one of~$f$ and~$f'$,
  say~$f$. Now expand the rectangle~$R(f)$ with a small parallel shift
  of the boundary segment containing~$e$. This makes the contact of~$R(f)$ and~$R(f')$ a true contact. The expansion can be taken small
  enough as to avoid that new contacts or intersections are introduced.
  Iterating over the edges we eventually get rid of all weak contacts. See \cref{fig:expand} for an illustration.
  
  In an orthogonal surface some flats might be co-planar. But by non-degeneracy all vertices are of degree three. Thus, flats (and the corresponding rectangles) can be perturbed slightly into the orthogonal direction such that co-planarity is avoided. This concludes the construction.
\end{proof}

Recall that we aim at realizing~$T^*$, the dual of the 3-colorable
triangulation~$T$ as the skeleton of an orthogonal surface.  Since~$T$ is
Eulerian its dual~$T^*$ is bipartite. Let~$U$ (black) and $U'$ (white) be the
bipartition of the vertices of~$T^*$ such that the dual~$v_\infty$ of the
outer face of~$T$ is in~$U$.  The critical task is to assign two extreme
vertices to each face of~$T^*$ such that~$v_\infty$ is never assigned.  This has to
be done so that each vertex in~$U$ (except~$v_\infty$) is extremal for exactly
two of the faces.

To solve the assignment problem we will work with an auxiliary
graph~$H_T$. The faces of~$T^*$ which do not contain~$v_\infty$ correspond to
the interior vertices of~$T$, we denote this set with~$V^\circ$.  As the
vertices of~$T^*$ are the facial triangles of~$T$, we think of~$U$ as
representing the black triangles of~$T$.  We also
let~$U^\circ = U - v_\infty$, this is the set of bounded black triangles
of~$T$.  The vertices of~$H_T$ are~$V^\circ\cup U^\circ$ the edges of~$H_T$
correspond to the incidence relation in~$T^*$ and~$T$ respectively,
i.e.,~$v,u$ with~$v\in V^\circ$ and~$u\in U^\circ$ is an edge if vertex~$v$ is
a corner of the black triangle~$u$. A valid assignment of extreme vertices is
equivalent to an orientation of~$H_T$ such that each vertex~$v\in V^\circ$ has
outdegree two and each vertex~$u\in U^\circ$ has indegree two, i.e., the
outdegrees of the vertices are prescribed by the function~$\alpha$
with~$\alpha(v)=2$ for~$v\in V^\circ$ and~$\alpha(u)=\deg(u)-2$
for~$u\in U^\circ$. Since~$|V^\circ|=|U^\circ|=n-3$ it is readily seen that
the sum of the $\alpha$-values of all vertices equals the number of edges
of~$H_T$.

Orientations of graphs with prescribed out-degrees have been studied
e.g.~in~\cite{f-lspg-04}, there it is shown that the following
necessary condition is also sufficient for the existence of an
$\alpha$-orientation.
For all~$W\subset V^\circ$ and~$S\subset U^\circ$ and~$X=W\cup S$
\begin{equation}\label{eqn:alpha}
\sum_{x\in X}\alpha(x) \leq |E[X]| + |E[X,\ovl{X}]|.\tag{$\alpha$}
\end{equation}
Here~$E[X]$ and~$E[X,\ovl{X}]$ denote the set of edges induced by
$X$, and the set of edges in the cut defined by~$X$, respectively.

Inequality~\eqref{eqn:alpha} does not hold for all triangulations~$T$ and
all~$X$. We next identify specific sets~$X$ violating the inequality, they are
associated to certain \textbf{ba}dly \textbf{be}having
\textbf{t}riangle\textbf{s}, which we will call {babets} for short. In
\cref{prop:babet} we then show that babets are the only obstructions for the
validity of~\eqref{eqn:alpha}.

Let~$\Delta$ be a separating triangle of~$T$ such that the faces of $T$
bounding~$\Delta$ from the outside are white. Let~$W$ be the set of vertices
inside~$\Delta$ and let~$S$ be the collection of black triangles of~$T$ which
have all vertices in~$W$. We claim that $X=W\cup S$ is
violating~\eqref{eqn:alpha}. If~$|W| = k$ and~$|S|=s$, then
$\sum_{x\in X}\alpha(x) = 2|W| + |S| = 2k + s$.  The triangulation whose outer
boundary is~$\Delta$ has~$2(k+3)-4$ triangles, half of them, i.e.,~$k+1$, are
black and interior.  The right side of~\eqref{eqn:alpha} is counting the
number of incidences between vertices of~$W$ and black triangles.  Black
triangles in~$\Delta$ have~$3(k+1)$ incidences in total. There are~$k+1 - s$
black triangles which have an incidence with a corner of~$\Delta$ and 3 of them have
incidences with two corners of~$\Delta$. Hence the value on the right side is
$3(k+1) - (k+1-s) -3 = 2k + s -1$.  This shows that the inequality is
violated.  A separating triangle~$\Delta$ of~$T$ with white touching triangles
on the outside is called \term{babet}.

\begin{proposition}\label{prop:babet}
  If~$T$ has no babet, then there is an orientation of~$H_T$
  whose outdegrees are as prescribed by~$\alpha$.
\end{proposition}

%%%%%%%%%%%%%%%%%%%%%%
%%%% REPHRASE ??? %%%%
%%%%%%%%%%%%%%%%%%%%%%
Before we prove \cref{prop:babet} and \cref{thm:planar}, let us briefly summarize the procedure.

First, we construct a Plattenbau representation in the babet-free case (\cref{ssec:babet-free}) based on an auxiliary graph~$G$ arising from the bipartition of~$T^*$, and a Schnyder wood~$S$ for~$G$.
We then find an orthogonal surface~$\SS$ based on the Schnyder wood~$S$ and show that the skeleton~$G_\SS$ of~$\SS$ is~$T^*$, which together with \cref{prop:OS-RC} gives a Plattenbau for~$T$.
Then (\cref{sec:patching}), in case~$T$ contains some babets, we cut the triangulation~$T$ along an innermost babet, find orthogonal surfaces for the orthogonal surfaces for the inside and outside, and patch the former into a saddle point of the latter.

Now, let us start with the proof of \cref{prop:babet}.

\begin{proof}
Suppose that there is an~$X=W\cup S$ violating
inequality~\eqref{eqn:alpha}. We are going to modify~$X$
in several steps always maintaining the property that the inequality is
violated. At the end we will be able to point to a babet in~$T$.

Suppose there is a~$u\in S$ with~$a\leq 2$ neighbors in~$W$.  Let
$X' = X-u$ when going from~$X$ to~$X'$ the left side
of~\eqref{eqn:alpha} is loosing~$\deg(u)-2$ while on the right side we
loose the~$\deg(u)-a$ edges of~$H_T$ which are incident to~$u$ but not to
$W$. Since~$a\leq 2$ the set~$X'$ is violating.
From now on we assume that every~$u\in S$ has 3 neighbors in~$W$,
in particular~$\alpha(u) = 1$.

Now the left side of~\eqref{eqn:alpha} equals~$2|W| + |S|$
and for the right side we have~$|E[X]|=3|S|$ and
$E[X,\ovl{X}]$ contains no edge incident to~$S$.
We define~$\partialW = |E[X,\ovl{X}]|$ the notation indicates that
we only have to care of boundary edges of~$W$. The assumption that
Inequality~\eqref{eqn:alpha} is violated 
then becomes~$2|W| + |S| > 3|S| + \partialW$ or equivalently
\begin{equation}\label{eqn:partial}
2|W| > 2|S|+\partialW.\tag{$\partial$}
\end{equation}

We can assume that the subgraph of~$H_T$ induced by~$X$ is connected,
otherwise a connected component would also violate.  Let~$|W|=k$ and
$|S|=s$. The set~$S$ is a set of black triangles in the triangulation $T$
and~$W$ is the set of vertices of these triangles. Let~$T_S$ be the plane
embedding of all edges of triangles of~$S$ as seen in~$T$. Classify the faces
of $T_S$ as black triangles, white triangles and big faces, and let their
numbers be $s$, $t$ and $g$, respectively. We consider the outer face of $T_S$
a big face independent of its size, therefore, $g\geq 1$.  Consider the
triangulation $T_S^+$ obtained by stacking a new vertex in each big face and
connecting it to all the angles of the face, i.e., the degree of the vertex
$z_f$ stacked into face $f$ equals the length $r_f$ of the boundary
of~$f$. Note that $T_S^+$ may have multi-edges but every face of $T_S^+$ is a
triangle so that Euler's formula holds. Let $R=\sum_f r_f$ be the sum of
degrees of the stack vertices. Since $T_S^+$ has $k+g$ vertices it has
$2(k+g)-4$ faces.  However, we also know that $T_S^+$ has $s+t + R$
faces. Counting the edges incident to the triangles of $S$ we obtain
$3s = 3t + R$. Using this to eliminate $t$ we obtain
$2(k+g)-4 = 2s + \frac{2R}{3}$, i.e., $2|W| = 2|S| + 4 - 2g +
\frac{2R}{3}$. With \eqref{eqn:partial} this implies
$4 + \frac{2R}{3} > \partialW + 2g$.

\begin{myclaim}  
 $R$ is divisible by 3.
\end{myclaim}
\begin{claimproof}
Actually we prove that each $r_f$ is divisible by 3. Let $Y$ be a
collection of triangles in a 3-colorable triangulation and let $\gamma$ be the
boundary of $Y$, i.e., $\gamma$ is the set of edges incident to exactly one
triangle from $Y$.  Let $Y_b$ and $Y_w$ be the black and white triangles in
$Y$ and let $\gamma_b$ and $\gamma_w$ be the edges of $\gamma$ which are
incident to black and white triangles of $Y$.  Double counting the number of
edges in $Y$ we get $3|Y_b| + |\gamma_w| = 3|Y_w| + |\gamma_s|$, hence,
$|\gamma_w| \equiv |\gamma_s| \mod{3}$. In our case $|\gamma_s|=0$ and
depending on the chosen $Y$ either $R = |\gamma_w|$ or $r_f = |\gamma_w|$.
\end{claimproof}

Let $\gamma$ be the boundary cycle of a big face of $T_S$ which is not the
outer face. We have seen that $|\gamma| \equiv 0 \mod{3}$. We are interested
in the contribution of $\gamma$ to $\partialW$, i.e., in the number of
incidences of vertices of $\gamma$ with black triangles in the inside of the
big face. For convenience we can use multiple copies of a vertex to make
$\gamma$ simple. We let $T_\gamma$ be the triangulation of the hole.  The
claim below implies that $\partial\gamma \geq 2|\gamma|/3$, unless
$|\gamma|\leq 6$ and there is at most one black triangle in $T_\gamma$.
Add all the inner black triangles of $\gamma$ to $S$ and all the inner
vertices to $W$ and consider the effect for the violator inequality
$4 + \frac{2R}{3} > \partialW + 2g$.  In the exceptional case we only have one
black triangle and $|\gamma|=6$, i.e., the left side is reduced by 4 and the
right side by $3+2$. In all other cases the left side of the
violator inequality is reduced by $2|\gamma|/3$ and the right side is reduced
by at least $2|\gamma|/3 + 2$, i.e., violators are preserved.

\begin{myclaim}  
  If $\gamma$ is a simple cycle in a 3-colorable triangulation with
  interior triangulation~$T_\gamma$ and
  $\partial\gamma$ is the number of incidences of vertices of $\gamma$ with
  black triangeles of $T_\gamma$, then
  \begin{itemize}
  \item $\partial\gamma \geq |\gamma| - 3$ if $T_\gamma$ has no interior vertex, and
  \item $\partial\gamma \geq 2|\gamma|/3$ otherwise.
  \end{itemize}
\end{myclaim}

\begin{claimproof}
  We assume that all the faces of $T_\gamma$ incident to an edge of $\gamma$ are white.
  If not then adding white triangles to achieve this property increases $|\gamma|$
  and keeps $\partial\gamma$ the same.
  
  Suppose $\partial\gamma < |\gamma|$.  Then on $\gamma$ we find an
  \textit{ear}, this is a vertex which has no incidence to a black triangle of
  $T_\gamma$, i.e., its degree in $T_\gamma$ is 2.

  We first deal with the case where $T_\gamma$ has no interior vertex, i.e.,
  all the edges not on $\gamma$ are chords.  If $b$ is an ear and $a$, $c$ are
  the neighbors of~$b$ on $\gamma$, then $ac$ is an edge and, if
  $|\gamma|>3$, there is a black triangle $acx$ in $T_\gamma$.  An ear is
  \textit{reducible} if~$b$ has a neighbor on $\gamma$ which is only incident
  to a single black triangle in $T_\gamma$. Assume that $c$ is such a neighbor
  of $b$, then the second neighbor $d$ of $c$ on $\gamma$ has an edge to $x$.
  delete $b$ and $c$ and identify $a$ with $d$ and also identify the edges
  $xa$ and $xd$. This results in a cycle $\gamma'$ with
  $|\gamma'| = |\gamma|-3$. We call this an \textit{ear reduction with center
    $x$}. \cref{fig:redu2} (left and middle) shows sketches of reducible ears
  with $x\in \gamma$.  If $|\gamma|>3$, then there is a reducible ear,
  otherwise the average degree of a vertex would be 4. This is impossible
  because $T_\gamma$ has $2 |\gamma|-3$ edges.

  If starting with $\gamma$ we can perform $m$ reductions, then $|\gamma| = 3m+3$
  and $\partial\gamma = 3m$. This completes the proof in this case.

  Now assume that $T_\gamma$ has an interior vertex. Again there is an ear
  $b$, let $a$, $c$ be the neighbors of~$b$ on $\gamma$. Then $a,c$ is an edge
  and there is a black triangle $acx$ in $T_\gamma$. Now $x$ may also be an
  inner vertex, see \cref{fig:redu2} (right). However, if $x$ is the unique
  common neighbor of $a$ and $d$, then we can perform an \textit{ear reduction
    with center $x$} by identifying $a$ with $d$ as well as the edges $xa$ and
  $xd$. This results in a cycle $\gamma'$ with $|\gamma'| = |\gamma|-3$.
  Note that when~$x$ is not on $\gamma$ we have
  $\partial\gamma = \partial\gamma' +2$, while
  $\partial\gamma = \partial\gamma' +3$ when $x$ belongs to $\gamma$.
  
%%%%%%%%%%%%%%%%%%%%%%%%%%%%%%%%%%%%%%%%%%%%%%%%%%%%%%%%%%%%%%%%%%%%%% 
% in einem figure environment mit caption
   \calc_figscale{50}
    \begin{figure}[htb]
    \centerline{\input{redu2.pdftex_t}}
    \caption{Three examples where $b$ is a reducible ear.\label{fig:redu2}}
    \end{figure}
    
%%%%%%%%%%%%%%%%%%%%%%%%%%%%%%%%%%%%%%%%%%%%%%%%%%%%%%%%%%%%%%%%%%%%%%

If $b$ is a reducible ear but $a$ and $d$ share
several neighbors, then there is an extreme common neighbor $y$ with the
property that the cycle $a,b,c,d,y$ encloses all the common neighbors of~$a$
and $d$.  In this case we perform an ear reduction with center $y$
(since $a$ and $d$ have the same color in every 3-coloring, the reduced graph remains
3-colorable and its faces 2-colorable). The reduction again yields a cycle $\gamma'$ with
$|\gamma'| = |\gamma|-3$ and $\partial\gamma \geq \partial\gamma' +2$.

After a series of $m$ reductions we obtain a cycle $\gamma'$ which has no
reducible ear. Suppose that there remains an interior vertex in the
triangulation $T_{\gamma'}$. Now every ear of $T_{\gamma'}$ has two neighbors
which are incident to at least two black triangles.  This shows that
$\partial\gamma' \geq |\gamma'|$. Hence
$\partial\gamma \geq \partial\gamma' + 2m \geq |\gamma'| +2m = |\gamma| -3m
+2m = |\gamma| -m$. Since $m \leq |\gamma|/3$ we arrive at
$\partial\gamma \geq 2|\gamma|/3$.

Now suppose that $T_{\gamma'}$ has no interior vertex. There has been a last
reduction where $a$ and~$d$ had at least two common neighbors and the
outermost $y$ was on the current outer cycle $\gamma^*$. In this case one of
the edges $ay$ or $dy$ is incident to a black triangle which disappears with
the reduction (cf.~\cref{fig:patch} (right)), hence, in this step
  $\partial\gamma^*$ drops by 4. For the initial $\gamma$ we get:
  $|\gamma| = 3m+3$ and $\partial\gamma \geq 2m+2$.
This completes the proof of the claim.
\end{claimproof}

We have already seen that the claim implies that we may assume 
that $g=1$, i.e.,  
the violator~$T_S$ only has a single big face, the outer face $f_\infty$. 
We write $r_{f_\infty}=3\rho$, the condition for violation becomes
$\partialW < 2\rho +2$.

Let $\gamma$ be the boundary cycle of the unique big face.  While the
outer face of $T$ has to be contained in the big face, we prefer to
think of a drawing of $T_S$ such that the big face is the interior of
$\gamma$.  It will be crucial, however, that in the inner
triangulation of $\gamma$ inherited from~$T$ there is a special black
triangle $\delta_\infty$. The goal is to find a babet in the interior
of $\gamma$. We use induction on $\rho$.

In the case $\rho=1$ the triangle $\gamma$ has $\partialW < 4$
incidences with black triangles from the inside. The unique
configuration with this property is shown in
\cref{fig:babet}. Since the black triangle $\delta_\infty$ is not yet
in the picture one of the triangles must be separating.
If the separating triangle was not the central one it would lead to an
increase of $\partialW$. Hence the white central triangle is separating,
i.e., a babet.

%%%%%%%%%%%%%%%%%%%%%%%%%%%%%%%%%%%%%%%%%%%%%%%%%%%%%%%%%%%%%%%%%%%%%% 
% in einem figure environment mit caption
   \calc_figscale{50}
    \begin{figure}[htb]
    \centerline{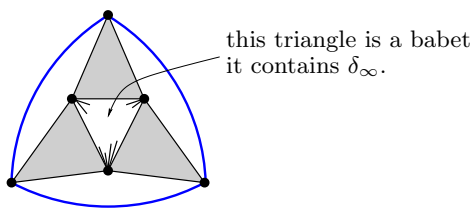}
    \caption{Illustration for the case $\rho=1$.\label{fig:babet}}
    \end{figure}
    
%%%%%%%%%%%%%%%%%%%%%%%%%%%%%%%%%%%%%%%%%%%%%%%%%%%%%%%%%%%%%%%%%%%%%%

Now let $\rho \geq 2$. Assuming a black incident triangle for every
vertex of $\gamma$ we obtain $\partialW \geq 3\rho \geq 2\rho+2$.
Therefore, on $\gamma$ we find an \textit{ear}.
As above we aim for an ear reduction. Let $b$ be a reducible ear with neighbors
$a$ and $c$ such that $c$ is only incident to a single black triangle $acx$
inside $\gamma$ and $d$ is the second common neighbor of $c$ and $x$.
In the proof of Claim 2 we have seen that in this case a reduction is
possible which preserves the violation inequality. 

If $x$ is the unique common neighbor of $a$ and $d$, then the reduction
leads to a decrease of $\partialW$ by 2 or 3 and a decrease of $\rho$ by 1,
whence the reduced cycle remains violating.

If $b$ is reducible but $a$ and $d$ share several neighbors, then we aim
at a reduction whose center $y$ is the extreme common neighbor of $a$ and
$d$. If this cycle does not enclose the black triangle $\delta_\infty$ we can
apply the reduction with center $y$.

The described reductions may decrease $\rho$ until it is 1 whence
there is a babet. In fact we will complete the proof by showing that
when no reduction is possible and $\rho > 1$, then the
violator inequality is not fulfilled.

We first discuss the case where $b$ is reducible but $a$ and $d$ share
several neighbors and $\delta_\infty$ is enclosed in the cycle
$a,b,c,d,y$ where $y$ is the extreme common neighbor of $a$ and $d$.
We show that in this case we can find a cycle $\gamma'$ of length 6
which is a violator, i.e., $\partial\gamma < 6$.

%%%%%%%%%%%%%%%%%%%%%%%%%%%%%%%%%%%%%%%%%%%%%%%%%%%%%%%%%%%%%%%%%%%%%% 
% in einem figure environment mit caption
   \calc_figscale{33}
    \begin{figure}[htb]
    \centerline{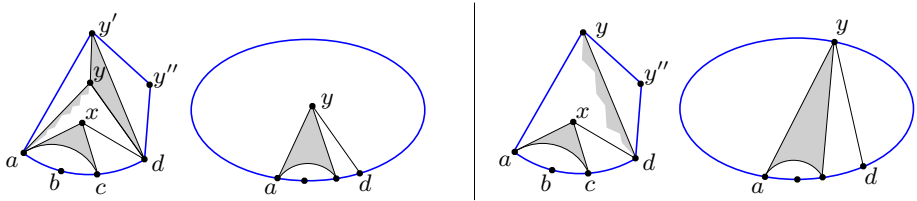}
    \caption{Illustration for the subconfigurations generated
from a cycle $a,b,c,d,y$ enclosing~$\delta_\infty$.\label{fig:patch}}
    \end{figure}
    
%%%%%%%%%%%%%%%%%%%%%%%%%%%%%%%%%%%%%%%%%%%%%%%%%%%%%%%%%%%%%%%%%%%%%%

Suppose $y$ is in the interior of $\gamma$ in this case we take the cycle
$a,b,c,d,y$ add a new common neighbor $y'$ for $a,d,y$ and an ear $y''$ over
the black triangle containing $y'$.  This yields a 6-cycle $\gamma''$ as shown
in \cref{fig:patch} (left 1). In the interior of $\gamma$ we replace the
triangles inside $a,b,c,d,y$ by the 3 triangles of a reducible ear, see
\cref{fig:patch} (left 2) and refer to the cycle with the simplified interior
as $\gamma'$.  Now we compare $\partial\gamma$ with $\partial\gamma'$ and
$\partial\gamma''$ and use the bound previously shown in the claim for
$\gamma'$, i.e., $\partial\gamma' \geq 2|\gamma'|/3 = 2|\gamma|/3$.  Taking
into account that on each of $\gamma'$ and $\gamma''$ we see two incidences
with black triangles which are not counted in $\partial\gamma$ we get
$2|\gamma|/3 + 2 > \partial\gamma = (\partial\gamma'' -2) +
(\partial\gamma'-2) \geq \partial\gamma'' + 2|\gamma|/3 -4$.  Hence
$\partial\gamma'' < 6$, whence $\gamma''$ is also a violator.

If $y$ is a vertex on the cycle $\gamma$ we add a new ear $y''$ either over
$ay$ or over $dy$ depending on which is incident to a black triangle. This
yields a 6-cycle $\gamma''$ as shown in \cref{fig:patch} (right 1).  In the
interior of $\gamma$ we replace the triangles inside $a,b,c,d,y$ by the 3
triangles of a reducible ear, see \cref{fig:patch} (right 2) and refer to the
cycle with the simplified interior as $\gamma'$.  Taking into account that on
$\gamma'$ we see three incidences with black triangles which are not counted
in~$\partial\gamma$ we get
$2|\gamma|/3 + 2 > \partial\gamma \geq \partial\gamma'' + (\partial\gamma'-3)
\geq \partial\gamma'' + 2|\gamma|/3 -3$.  Hence $\partial\gamma'' < 5$, whence
$\gamma''$ is also a violator.

If $\rho \geq 3$ and there are no reducible ears, then both
neighbors of each ear vertex have two incidences with black triangles. Let each
vertex with at least two incidences with black triangles
discharge~$1/2$ to the neighbors, then all vertices
have a weight of at least~1. We get $\partialW \geq 3\rho > 2\rho +2$
and the example is not a violator.

For the case $\rho=2$ we consider circular sequences
$(s_1,s_2,s_3,s_4,s_5,s_6)$ with $\sum_i s_i < 2\rho+2=6$ such that
there is an inner Eulerian triangulation of a 6-gon with only white
triangles touching the 6-gon and $s_i$ black incident triangles at
vertex $v_i$ of the 6-gon. A vertex $v_i$ with $s_i=0$ is an ear.
Clearly, the circular sequence has no $00$ subsequence.  With two ears
which share a neighbor of degree 1, i.e., with a $010$ subsequence,
we have the graph $T_1$ from \cref{fig:6triang}, this triangulation does not
occur in our setting since it does not contain a black $\delta_\infty$
which is vertex disjoint from the outer 6-gon. Making any of the four
triangles of $T_1$ separating so that it can accommodate
$\delta_\infty$ in its interior would make $\sum_i s_i \geq
6$.

%%%%%%%%%%%%%%%%%%%%%%%%%%%%%%%%%%%%%%%%%%%%%%%%%%%%%%%%%%%%%%%%%%%%%% 
% in einem figure environment mit caption
   \calc_figscale{50}
    \begin{figure}[htb]
    \centerline{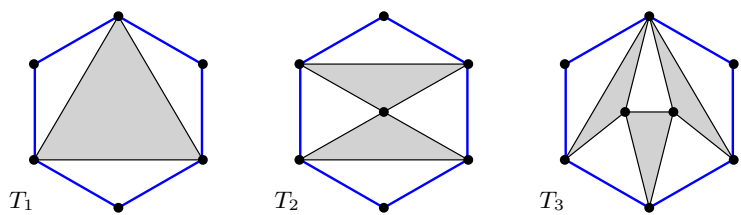}
    \caption{Three triangulations with $\rho=2$ and $\partialW < 6$.\label{fig:6triang}}
    \end{figure}
    
%%%%%%%%%%%%%%%%%%%%%%%%%%%%%%%%%%%%%%%%%%%%%%%%%%%%%%%%%%%%%%%%%%%%%%

With a $0110$ subsequence we have $T_2$. Again there is no
$\delta_\infty$ in $T_2$ and making a triangle separating would make
$\sum_i s_i \geq 6$.  It remains to look at sequences without $00$, and
$010$, and $0110$ but $\sum_i s_i \leq 5$. The sequence $011102$ is the
unique sequence with these properties and there is a unique
corresponding triangulation $T_3$. As with $T_1$ and $T_2$ there is no
$\delta_\infty$ in $T_3$ and making a triangle separating would make
$\sum_i s_i \geq 6$.
\end{proof}

\subsection{Plattenbau Representations in the Babet-Free Case}
\label{ssec:babet-free}

Let $T$ be a 3-colorable triangulation which has no babet.
Due to \cref{prop:babet} we find the $\alpha$-orientation of
$H_T$. This orientation can be represented on the
dual $T^*$ of $T$ as a collection $\CC$ of cycles such that every face
which is not incident to $v_\infty$ contains exactly one fragment of a
cycle which leaves the face in distinct vertices of $U^\circ$ and
every vertex of $U^\circ$ is covered by one of the cycles (Eppstein
and Mumford~\cite{EppM-14} refer to this structure as \textit{cycle
  cover}). The connected components of $\RR^2\setminus\CC$ can be
two-colored in white and pink such that the two sides of each cycle
have distinct colors, the two-coloring is unique if we want the
unbounded region to be colored white. This two-coloring of the plane
induces a partition of the vertices of $U'$ into $U'_w,U'_p$ where
$U'_w$ consists of all vertices living in white regions and $U'_p$
consists of the vertices in pink regions, see \cref{fig:exa-2factor}
for an example.  An important property of the partition is that every
vertex of $u\in U^\circ$ is adjacent to vertices from both classes
$U'_w$ and $U'_p$.  This follows form the fact that a cycle $C$ from
$\CC$ contains $u$, whence two of the edges of $u$ are on one side and
the third is on the other side of $C$.

%%%%%%%%%%%%%%%%%%%%%%%%%%%%%%%%%%%%%%%%%%%%%%%%%%%%%%%%%%%%%%%%%%%%%% 
% in einem figure environment mit caption
   \calc_figscale{25}
    \begin{figure}[htb]
    \centerline{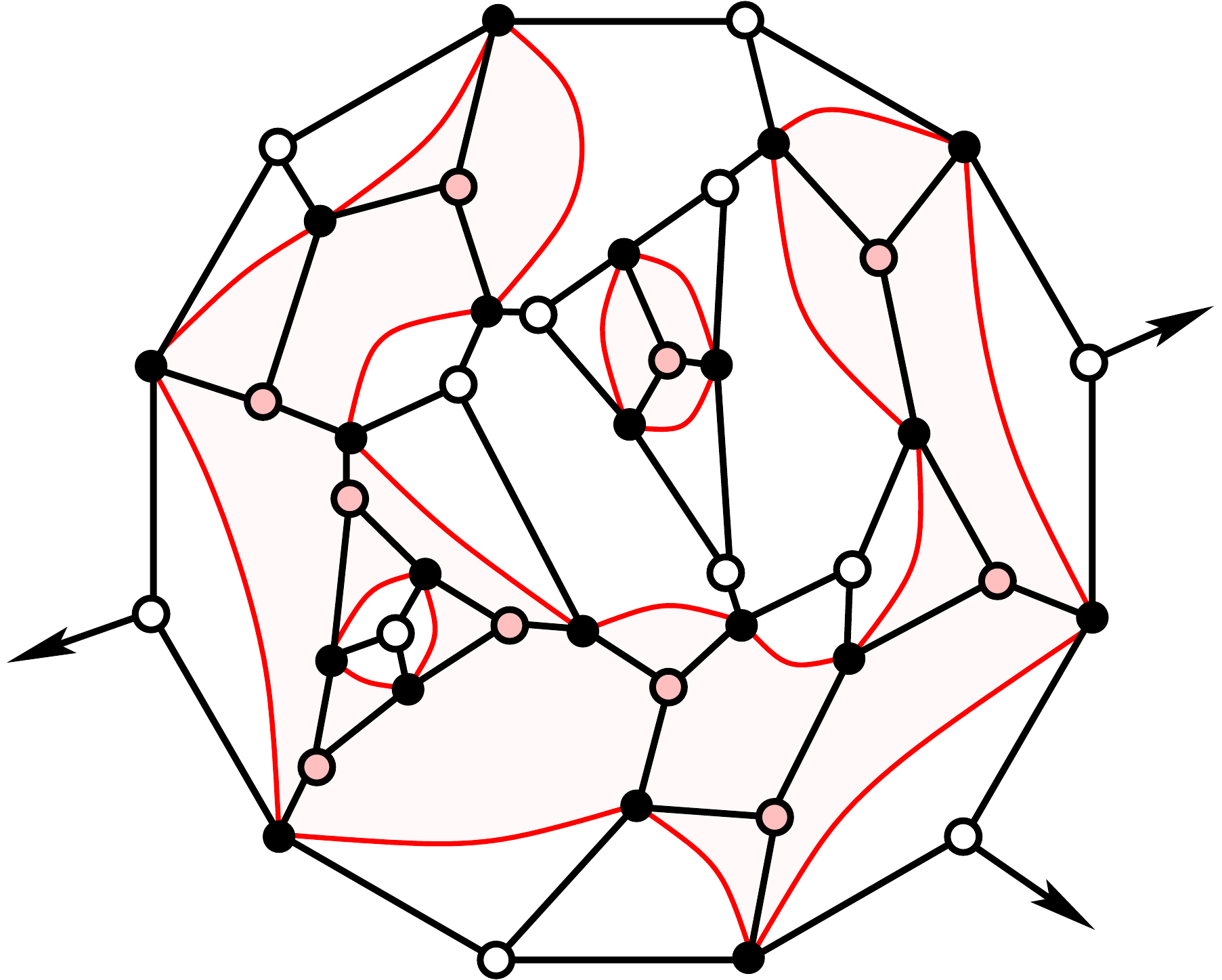}
    \caption{Example of a $T^*$ with a cycle cover of
  $U^\circ$ and the induced partition of $U'$.\label{fig:exa-2factor}}
    \end{figure}
    
%%%%%%%%%%%%%%%%%%%%%%%%%%%%%%%%%%%%%%%%%%%%%%%%%%%%%%%%%%%%%%%%%%%%%%

The partition $U'_w,U^\circ,U'_p$ of the vertices of
$T^*\setminus\{v_\infty\}$ corresponds to the partition into minima, saddle points, maxima
of the vertices of the orthogonal surface with skeleton $T^*$.  To
construct this orthogonal surface we first define a 3-connected planar
graph $G$ whose vertex set is $U'_w$, the edges of $G$ are in
bijection with $U^\circ$ and each bounded face of $G$ contains exactly
one vertex of $U'_p$. This graph $G$ will be decorated with a Schnyder
wood, i.e., an orientation of the edges which obeys the
following rules:

\begin{enumerate}[leftmargin=2.8em,label = (W\arabic*)]
\item On the outer face of $G$ there are three special vertices
  $a_r,a_g,a_b$ colored red, green, and blue in clockwise order.  For
  $c\in\{r,g,b\}$ vertex $a_c$ is equipped with an outward oriented
  half-edge of color $c$. \label{item:W1}
\item
Every edge $e$ is oriented in one or in two opposite directions.
The directions of edges are colored such that
if $e$ is bioriented the two directions have
distinct colors. \label{item:W2}
\item
Every vertex $v$ has outdegree one in each color $c\in\{r,g,b\}$.
The edges $e_r,e_g,e_b$ leaving $v$ in colors $r,g,b$ occur in clockwise
order. Each edge entering $v$ in color $c$ enters $v$ in the 
sector bounded by the two $e_i$ with colors different from $c$
(see \cref{fig:node-prop}). \label{item:W3}
\item
There is no interior face whose boundary is a directed cycle
in one label. \label{item:W4}
\end{enumerate}
%%
%%%%%%%%%%%%%%%%%%%%%%%%%%%%%%%%%%%%%%%%%%%%%%%%%%%%%%%%%%%%%%%%%%%%%% 
% in einem figure environment mit caption
   \calc_figscale{25}
    \begin{figure}[htb]
    \centerline{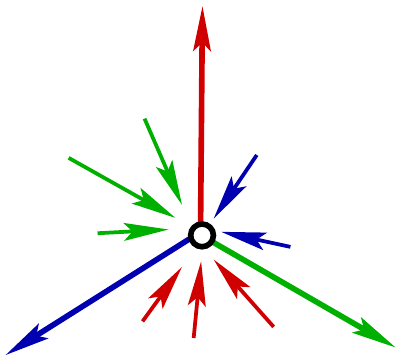}
    \caption{Illustration for the vertex condition \ref{item:W3}.\label{fig:node-prop}}
    \end{figure}
    
%%%%%%%%%%%%%%%%%%%%%%%%%%%%%%%%%%%%%%%%%%%%%%%%%%%%%%%%%%%%%%%%%%%%%%
%%
The construction of $G$ and the Schnyder wood $S$ is in several steps. Consider a 3-coloring
of $T$ with colors $r,g,b$, such that on the outer triangle these colors appear clockwise in the
given order. Note that this implies that all white triangles see $r,g,b$ in
clockwise order and all black triangles see $r,g,b$ in counterclockwise order.
The coloring of $T$ induces a 3-coloring of the edges: for edge $e=v,v'$ use
the unique color which is not used for $v$ and $v'$. This edge-coloring of $T$
can be copied to the dual edges. This yields an edge-coloring of $T^*$ such
that each vertex in $U'$ is incident to edges of colors $r,g,b$ in clockwise
order.  Delete $v_\infty$ from $T^*$ but keep the edges incident to $v_\infty$
as half edges at their other endpoint, we denote the obtained graph as
$T^*_\infty$. The neighbors of~$v_\infty$ are the special vertices
$a_r,a_g,a_b$ for the Schnyder wood.

Next we introduce some new edges. For every black vertex $u\in U^\circ$
which is adjacent to only one vertex white vertex $v$ of $U'_w$ we identify the
unique face $f$ which is adjacent to $u$ but not to $v$. Connect
$u$ to a vertex $v'$ on the boundary of $f$ which belongs to $U'_w$.
Note that the edge $uv'$ is intersected by a cycle from $\CC$ and
that both faces obtained by cutting $f$ via $u,v'$ contain the vertex
$v' \in U'_w$. This shows that we can add edges to all 
vertices of $U^\circ$ which are adjacent to only one vertex in $U'_w$
without introducing crossings. The color of $vu$ is copied to the new
edge $uv'$.
\cref{fig:exa-pdSW} exemplifies the coloring of $T^*_\infty$ together
with the additional edges.

%%%%%%%%%%%%%%%%%%%%%%%%%%%%%%%%%%%%%%%%%%%%%%%%%%%%%%%%%%%%%%%%%%%%%% 
% in einem figure environment mit caption
   \calc_figscale{21}
    \begin{figure}[htb]
    \centerline{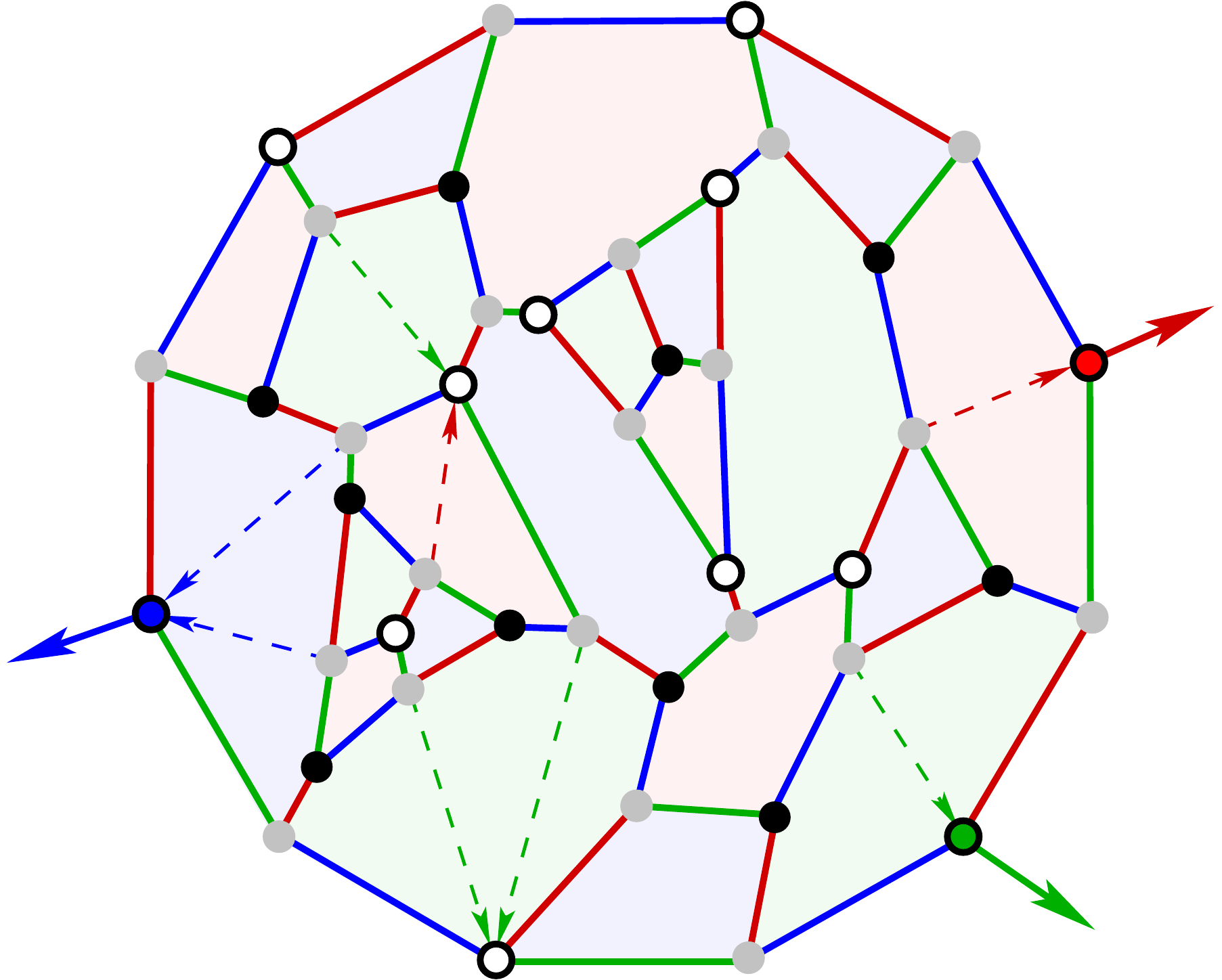}
    \caption{The coloring of the edges of $T^*_\infty$ obtained from
  the 3-coloring of $T$. The dashed arrows are the additional edges.\label{fig:exa-pdSW}}
    \end{figure}
    
%%%%%%%%%%%%%%%%%%%%%%%%%%%%%%%%%%%%%%%%%%%%%%%%%%%%%%%%%%%%%%%%%%%%%%

Now remove all the edges incident to vertices in $U'_p$. In the remaining
graph all the vertices of $U^\circ$ are of degree~2.  We remove these
`subdivision' vertices and melt the two edges into one. The result is $G$. We
claim that orienting the three edges of $v\in U'_p$ which come from an edge of
$T^*_\infty$ as outgoing we obtain a Schnyder wood of $G$. Indeed
\ref{item:W1}, \ref{item:W2}, and \ref{item:W3} follow directly or from what
we have already said. For \ref{item:W4} we need a little argument. Consider a
monochromatic edge $vv'$ and let $f$ be one of the faces of $G$ containing
this edge on the boundary. The other edge on the boundary of $f$ which
contains $v'$ is either incoming in the same color or outgoing in a different
color. This shows that there is no monochromatic directed facial cycle
containing $vv'$. Now consider a face $f$ of $G$ which has no monochromatic
edge. Note that $f$ is a union of faces of $T^*_\infty$ and each face of
$T^*_\infty$ is incident to a pink vertex in $U'_p$. Let $w\in U'_p$ be a
vertex in the interior of $f$ such that in $T^*_\infty$ there is an edge $wu$
with $u\in U^\circ$ on the boundary $\partial f$ of $f$.  Let $c\in\{r,g,b\}$
be the color of the edge $wu$. Vertex $u$ has two neighbors $v,v'$ on
$\partial f$.  By looking at the colors of edges of $T^*_\infty$ we see that
at $v$ and $v'$ the incident edge on $\partial f$ which is different from
$vv'$ has color $c$. In the Schnyder wood we therefore see these outgoing in
color $c$ at $v$ and $v'$ one of them being oriented clockwise, the other
counterclockwise on the boundary of $f$.  This shows that $\partial f$
supports no monochromatic directed cycle.

For the following we rely on the theory of Schnyder woods for
3-connected planar graphs, see e.g.\ \cite{f-cdpgo-01} or
\cite{fz-os3d-06} or \cite{f-lspg-04}. In the Schnyder wood $S$ on $G$ for
every vertex $v$ and every color $c\in\{r,g,b\}$ there is a directed
path $P_c(v)$ of color $c$ from $v$ to $a_c$.  The three paths
$P_r(v),P_g(v),P_b(v)$ pairwise only share the vertex $v$. The three
paths of $v$ partition the interior of $G$ into 3 regions. For
$\{c_1,c_2,c_3\} = \{r,g,b\}$ we let $R_{c_1}(v)$ be the region
bounded by $P_{c_2}$ and $P_{c_3}$. With $v$ we associate the region
vector $(v_r,v_g,v_b)$ where $v_c$ is the number of faces of $G$ in
region $R_{c}(v)$. \cref{fig:exa-pdSW} illustrates regions and region vectors.
%%
%%%%%%%%%%%%%%%%%%%%%%%%%%%%%%%%%%%%%%%%%%%%%%%%%%%%%%%%%%%%%%%%%%%%%% 
% in einem figure environment mit caption
   \calc_figscale{30}
    \begin{figure}[htb]
    \centerline{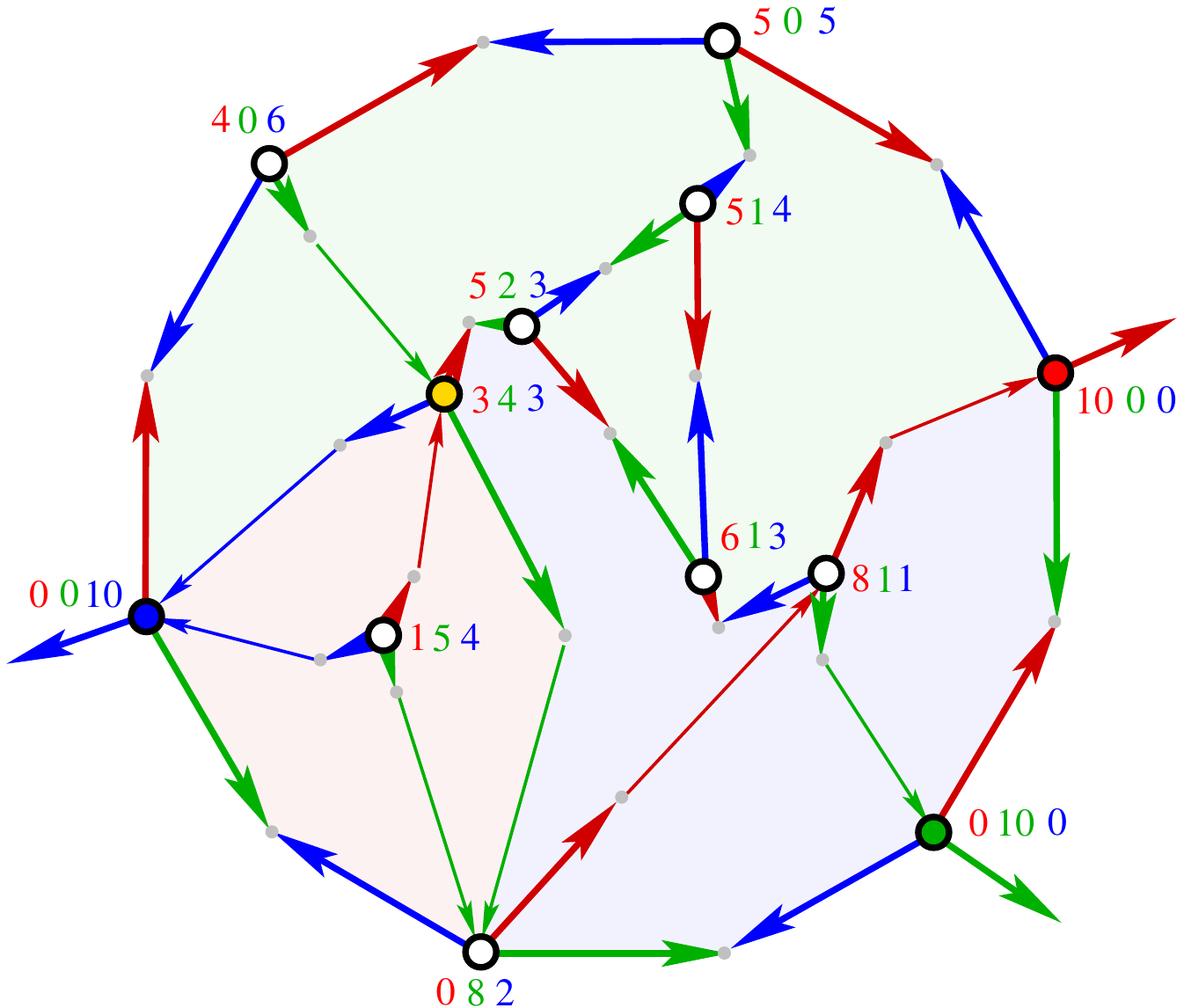}
    \caption{A Schnyder wood with a shading indicating the
  regions of the yellow vertex. Vertices are labeled with their
  region vector.\label{fig:exa-G-SW}}
    \end{figure}
    
%%%%%%%%%%%%%%%%%%%%%%%%%%%%%%%%%%%%%%%%%%%%%%%%%%%%%%%%%%%%%%%%%%%%%%

Let $\VV= \{(v_r,v_g,v_b) \;\colon\; v\in V(G)\}$ be the generating
set for an orthogonal surface $\SS$. In slight abuse of notation we
identify region vectors with their corresponding vertices and say that
$\SS$ is generated by $V(G)$. The minima of $\SS$ are the vertices of
$G$. Moreover, $\SS$ supports the Schnyder wood $S$ in the sense that
every outgoing edge at $v$ in $S$ corresponds to an edge of the
skeleton of $\SS$ such that the direction of the skeleton edge is
given by the color of the edge. In fact from the clockwise order
of directions of skeleton edges at minima, saddle points and maxima
it can be concluded that the skeleton $G_\SS$ of $\SS$ is $T^*$.
With \cref{prop:OS-RC} we obtain a Plattenbau representation of $T$.

%%%%%%%%%%%%%%%%%%%%%%%%%%%%%%%%%%%%%%%%%%%%%%%%%%%%%%%%%%%%%%%%%%%%%% 
% in einem figure environment mit caption
   \calc_figscale{45}
    \begin{figure}[htb]
    \centerline{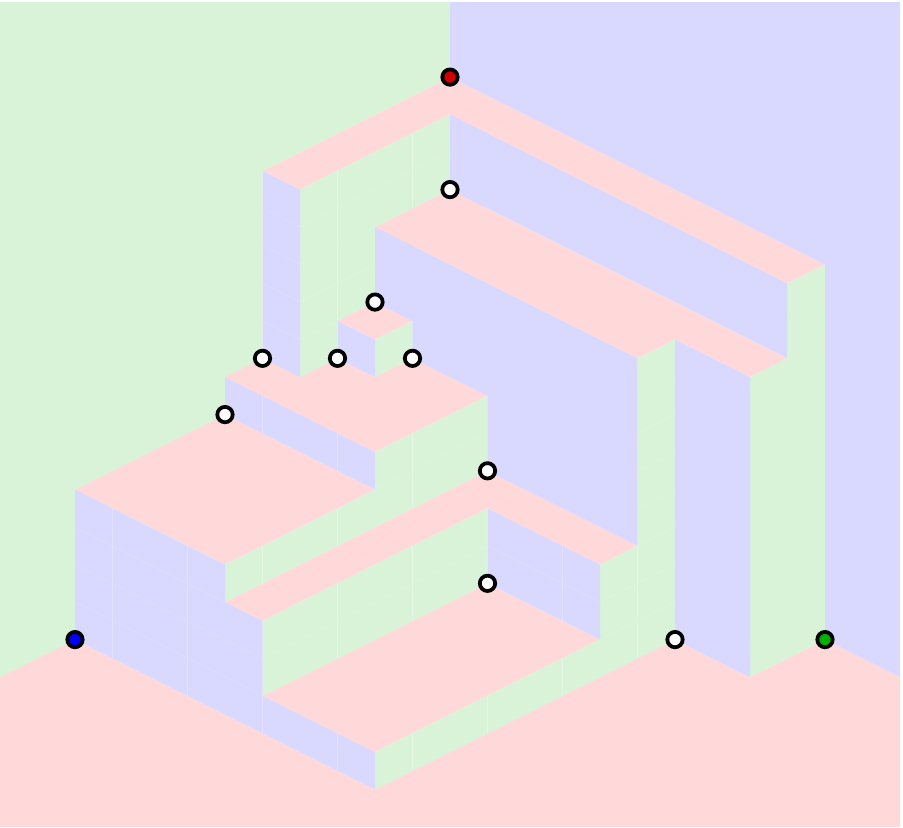}
    \caption{The orthogonal surface obtained from the
  region vectors given in \cref{fig:exa-G-SW}.\label{fig:exa-OS}}
    \end{figure}
    
%%%%%%%%%%%%%%%%%%%%%%%%%%%%%%%%%%%%%%%%%%%%%%%%%%%%%%%%%%%%%%%%%%%%%%

Since the set of $\alpha$-orientations of a fixed planar graph carries the
structure of a distributive lattice~\cite{f-lspg-04}, and we have shown a
correspondence of such a set with the orthogonal surfaces with given skeleton,
we obtain:

\begin{corollary}\label{cor:dist}
  Let $G$ be a planar cubic bipartite graph with specified vertex $v_{\infty}$
  such that the dual triangulation is babet-free. The set of orthogonal
  surfaces with skeleton $G$ and vertex $v_{\infty}$ at infinity carries a
  distributive lattice structure.
\end{corollary}

\subsection{Plattenbau Representations in the Presence of Babets}
\label{sec:patching}
\def\triang{\vartriangle}

Let $T$ be a 3-colorable triangulation, suppose that $T$ contains
babets. Being separating triangles babets can be nested, let $\BB$ be
the family of basic babets of $T$, i.e., of babets which are not contained in
the interior of another babet. Let $T_\triang$ be the triangulation
obtained from $T$ by cleaning all the babets, i.e., removing the
interior vertices and their incident edges from all babets $B\in
\BB$. Clearly $T_\triang$ is 3-colorable and babet-free. Triangles
which have been babets are black. With the method from the previous
subsection we get an orthogonal surface $\SS_\triang$ for $T_\triang$. In this
representation triangles which have been babets correspond to saddle
points. For later reference let $u_B$ be the saddle point
corresponding to $B\in \BB$.

For each babet $B\in \BB$ let $T_B$ the inside triangulation of $B$ in
$T$. Clearly, $T_B$ is 3-colorable, hence, its triangles can be
colored black and white with the outer face being black (note that this
coloring of $T_B$ is obtained from the coloring of triangles in $T$
by exchanging black and white). Assuming that $T_B$ has no babet
we obtain an orthogonal surface $\SS_B$ for $T_B$. 

The construction of the orthogonal surface $\SS_B$ works with the
assumption that the vertices of the outer face, i.e., of the triangle
$B$ are colored $r,g,b$ in clockwise order. The same assumption for
the full triangulation $T$ implies that the vertices of $B$ are
colored $r,b,g$ in clockwise order. 

The goal is to patch $\SS_B$ at the saddle point $u_B$ to
$\SS_\triang$ so that the flats corresponding to a vertex
of $B$ in the two orthogonal surfaces are coplanar.
Let $f_r$, $f_g$, and $f_b$ be the red, green and blue flat at $u_B$
in $\SS_\triang$, they represent the vertices $v_r,v_g,v_b$ of $B$
with their color in $T$. At $u_B$ exactly one of the three flats has a
concave angle, we assume that this flat is $f_r$, the other cases are
completely symmetric. The point $u_B$ is an interior point of the
rectangle $R = R(f_r)$ spanned by the extreme points of $f_r$.
The point $u_B$ is the apex of a convex corner whose sides
coincide locally with $R\setminus f_r$, $f_g$ and $f_b$.
In this corner we see the colors of the flats in clockwise order
as $r,b,g$. Hence, we can patch and appropriately scaled down
copy of $\SS_B$ into this corner
such that the red outer flat of $\SS_B$ becomes part of $R\setminus
f_r$, while the green outer flat of $\SS_B$ becomes part of $f_b$
and the blue outer flat of $\SS_B$ becomes part of $f_g$.
With the technique of \cref{prop:OS-RC} we obtain a rectangle
contact representation of the subgraph of $T$ induced by all the
vertices of $T$ which are represented by flats of
$\SS_\triang$ and $\SS_B$. See \cref{fig:patching} for an illustration.

%%%%%%%%%%%%%%%%%%%%%%%%%%%%%%%%%%%%%%%%%%%%%%%%%%%%%%%%
 \begin{figure}[htb]
  \centering
  \includegraphics[width =.49\textwidth]{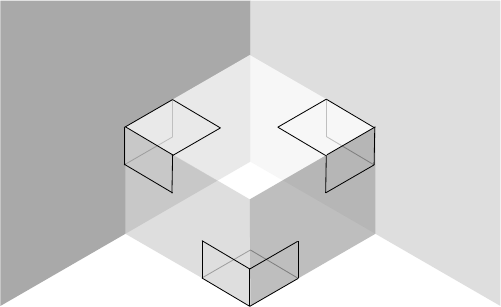}
  \caption{An orthogonal surface with three saddle points. The frames indicate how to patch
    a small orthogonal surfaces at the respective saddle points.}
  \label{fig:patching}
 \end{figure}
%%%%%%%%%%%%%%%%%%%%%%%%%%%%%%%%%%%%%%%%%%%%%%%%%%%%%%%%

Repeating the procedure for further babets in $\BB$ and for babets
which may occur in triangulations $T_B$ we eventually obtain an
rectangle contact representation of $T$.
This completes the proof of \cref{thm:planar}.

\subsection{Comments on Related Work}
\label{ssec:EM+G}

%%%%%%%%%%%%%%%%%%%%%%%%%%%%%%%%%%%%%%%%%%%
%%%% this is the short version from WG %%%%
%%%%%%%%%%%%%%%%%%%%%%%%%%%%%%%%%%%%%%%%%%%
% Additionally, in \cref{ssec:EM+G} we discuss the relation of our proof to two results in the literature:
% (1) The characterization of skeleton graphs of corner polyhedra due to Eppstein and Mumford~\cite{EppM-14}, and 
% (2) the TC-schemes for 3-colorable triangulations due to Gon\c{c}alves~\cite{g-3cpghaicru3s-19}.
%
% The set of $\alpha$-orientations of a fixed planar graph carries the
% structure of a distributive lattice~\cite{f-lspg-04}. With \cref{prop:babet}
% we can establish a correspondence of such a set with the orthogonal surfaces
% having a given skeleton. We obtain that this set carries such a structure,
% see \cref{cor:dist}.

Eppstein and Mumford~\cite{EppM-14} study orthogonal polytopes and their
graphs.  They define \term{corner polyhedra} as polytopes obtained from what
we call an orthogonal surface by restricting the surface to the bounded flats
and connecting their boundary to the origin $\mathbf{0}=(0,0,0)$, this
replaces the 3 unbounded flats of an orthogonal surface by three flats closing
the polytope. Eppstein and Mumford show that the skeleton graphs of corner
polyhedra are exactly the cubic bipartite 3-connected graphs with the
property that every separating triangle of the planar dual graph has the same
parity. This is equivalent to our characterization of these graphs as duals of
3-colorable triangulations with admit a choice of the outer face such that
there are no babets, see~\cref{prop:babet}.

A major part of their proof is devoted to the construction of a rooted cycle
cover, respectively, to the investigation of necessary and sufficient
conditions for the existence of such a cycle cover. The proof is based on a
set of operations that allow any 4-connected Eulerian triangulation to be
reduced to a smaller one. In contrast we show the equivalent existence of an
$\alpha$-orientation with a counting argument. Given a cycle cover Eppstein
and Mumford provide a construction of an appropriate orthogonal surface from
the combinatorial data by combining the coordinates obtained from plane
drawings of three orthogonal projections. In contrast we refer to the
established theory of Schnyder woods and their relation to orthogonal
structures to get the results.

As a more general class than corner polyhedra Eppstein and Mumford define
\term{xyz polyhedra} as orthogonal polytopes with the property that each
axis-parallel line through a vertex contains exactly one additional vertex.
They characterize the skeletons of them as cubic bipartite 3-connected graphs,
i.e., as the duals of 3-colorable triangulations.
Modulo the `reflection' of the three outer flats this corresponds to our
\cref{thm:planar}. The main step in the proof is the gluing of
an orthogonal surface into a corner of another orthogonal surface,
see \cref{sec:patching} and also Fig.~29 in~\cite{EppM-14}. 
\medskip

A recent paper of Gon\c{c}alves~\cite{g-3cpghaicru3s-19} can be used as a
basis for yet another proof of~\cref{thm:planar}, i.e., a proof of the
characterization of xyz polyhedra. Gon\c{c}alves uses a system of linear
equations to construct a \term{TC-scheme} (triangle contact scheme) for a
given 3-colorable triangulation. The TC-scheme comes very close to a segment
contact representation with segments of 3 slopes for the input graph, however,
there can be degeneracies: segments may degenerate to points (this relates to
babets) and segments ending on two sides of another segment may have
coinciding endpoints. The TC-scheme can be transformed into an orthogonal
surface. First adjust the directions to have slopes $0$, $+\frac{\pi}{3}$, and
$-\frac{\pi}{3}$, then add an orthogonal peak in each gray
triangle\footnote{This refers to the two color classes of the triangles of the
  TC-scheme, not to the two classes of triangles of the original
  triangulation.}  and an orthogonal valley in each white triangle, and extend
the outer flats. This yields an orthogonal surface. If there are no
degeneracies the orthogonal surface properly represents the triangulation via
flat contacts. Degeneracies of the TC-scheme translate into corners of degree
6 in the orthogonal surface, they can be resolved by \emph{shifting flats}
(c.f.~\cite{fz-os3d-06} for details on flat shifting). Finally as in the other
two approaches babets have to be recovered by patching their orthogonal
surface into corners of the surface.

%%%%%%%%%%%%%%%%%%%%%%%%%%%%%%%%%%%%%%%%%%%%%%%%%%%%%%%%%%%%%%%%%%%%%% 
% in einem figure environment mit caption
   \calc_figscale{12}
    \begin{figure}[htb]
    \centerline{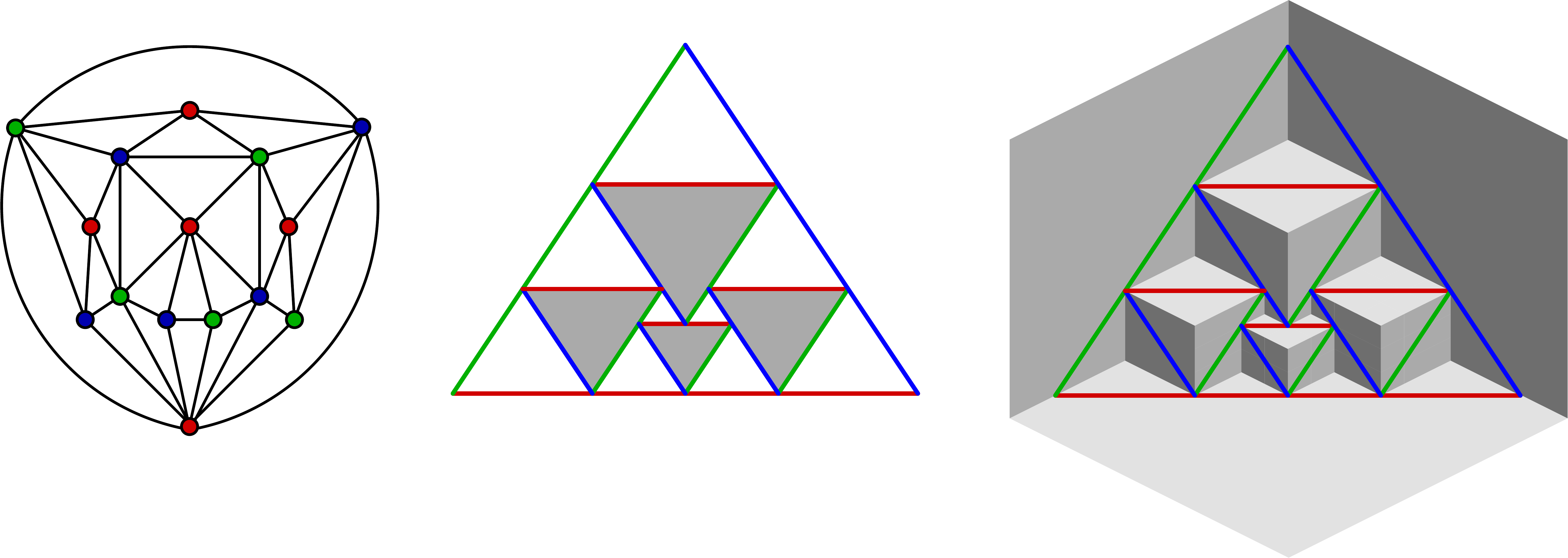}
    \caption{A 3-colored triangulation, a TC-scheme of
  the triangulation and the corresponding orthogonal surface.\label{fig:tc-scheme}}
    \end{figure}
    
%%%%%%%%%%%%%%%%%%%%%%%%%%%%%%%%%%%%%%%%%%%%%%%%%%%%%%%%%%%%%%%%%%%%%%

A nice aspect of this approach is that the partition of white
triangles of the triangulation into peak and valley triangles is
done by solving the linear system, no need of computing an
$\alpha$-orientation or a cycle cover for this task.

%%%%%%%%%%%%%%%%%%%%%%%%%%%%%
%%                         %%
%%        SECTION 4        %%
%%                         %%
%%%%%%%%%%%%%%%%%%%%%%%%%%%%%
\section{Generic Boxed Plattenbauten and Octahedrations}
\label{sec:proper-boxed}
 
In this section we characterize the touching graphs of generic boxed
Plattenbauten, that is, we prove \cref{thm:proper-boxed}. Furthermore, we
provide iterative constructions for generic boxed Plattenbauten from smaller
ones in Subsection~\ref{subsec:iterative}. Here is the theorem again.

\OCTA*

First, as noted earlier, in a generic Plattenbau for any two touching
rectangles~$R,R'$ one rectangle, say~$R$ uses part of one of its edges for
this incidence and does not use this edge for any other incidence.  We denote
this as~$R \to R'$ and remark that this orientation has already been used
in the proof of \cref{obs:proper-gives-sparse}.

Second, in any generic boxed Plattenbau~$\R$ there are six rectangles that are
incident to the unbounded region.  We refer to them as \term{outer rectangles}
and to the six corresponding vertices in the touching graph~$G$ for~$\R$ as
the \emph{outer vertices}.  The corners incident to three outer rectangles are
the \term{outer corners}, and the inner regions/cells of~$\R$ will be called
\term{rooms}.

Whenever we have specified some vertices of a graph to be outer vertices, this defines inner vertices, outer edges, and inner edges as follows:
The \emph{inner vertices} are exactly the vertices that are not outer vertices; the \emph{outer edges} are those between two outer vertices; the \emph{inner edges} are those with at least one inner vertex as endpoint.
We shall use these notions for a Plattenbau graph, as well as for some planar quadrangulations we encounter along the way.

% A \emph{spherical quadrangulation} is a graph $Q$ embedded on the 2-dimensional sphere without crossings with all faces bounded by a 4-cycle in $Q$.
% Spherical quadrangulations are 2-connected, planar, and bipartite.
% We will consider orientations of $G$, where some edges are \emph{bidirected}, i.e., both half-edges point towards the center of the edge.
% A bidirected edge contributes outdegree $1$ to both of its endpoints.

Let us start with the necessity of \crefrange{P1:outer-octahedron}{P4:common-neighbors} in \cref{thm:proper-boxed}.

\begin{proposition}\label{prop:boxed-necessity}
 Every touching graph of a generic boxed Plattenbau satisfies \crefrange{P1:outer-octahedron}{P4:common-neighbors} in \cref{thm:proper-boxed}.
\end{proposition}
\begin{proof}
 \cref{P1:outer-octahedron} follows directly from the definition of boxed Plattenbauten.
 Also, for \cref{P2:4-orientation} simply orient each edge towards the end-vertex whose rectangle has interior points in the intersection.
 This way, any edge between two outer vertices is bidirected, which is in accordance with \cref{P2:4-orientation}.
 Now, \cref{P3:quadrangulation} follows from \cref{obs:easy}~\cref{enum:planar} together with the fact that edge-maximal planar bipartite graphs are quadrangulations.
 Indeed, if the rectangles on one side of a given rectangle would not induce a quadrangulation, then there would be a rectangle with a free edge and $\R$ would not be boxed.
 Let us finally argue \cref{P4:common-neighbors}.
 The common neighbors of two vertices $u,v$ lie on the circle that is the intersection of the spheres with centers $SQ(u)$ and $SQ(v)$, respectively.
 Since $u$ and $v$ are on different sides of the intersection, $u,v$ see the vertices on the circle in opposite order.  
 See \cref{fig:common-neighborhood} for an illustration.
\end{proof}

%%%%%%%%%%%%%%%%%%%%%%%%%%%%%%%%%%%%%%%%%%%%%%%%%%%%%%%%
\begin{figure}
 \centering
 \includegraphics[width=0.5\textwidth]{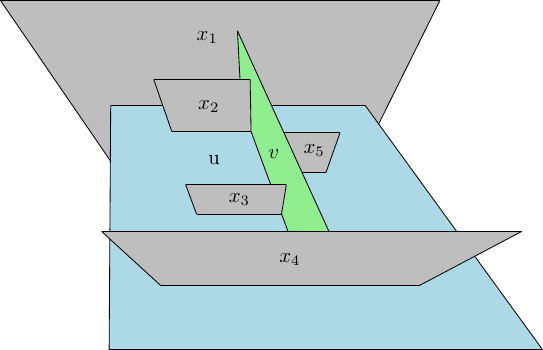}
 \hspace{2em}
 \includegraphics{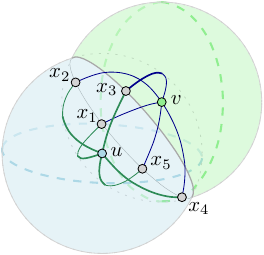} 
 \caption{Two rectangles (green and blue) and their common neighborhood.}
 \label{fig:common-neighborhood}
\end{figure}
%%%%%%%%%%%%%%%%%%%%%%%%%%%%%%%%%%%%%%%%%%%%%%%%%%%%%%%%

Next, we prove the sufficiency in \cref{thm:proper-boxed}, i.e., for every graph~$G$ satisfying \crefrange{P1:outer-octahedron}{P4:common-neighbors} we find a generic boxed Plattenbau with touching graph~$G$.

Fix a graph~$G = (V,E)$ with six outer vertices and edge orientation fulfilling \crefrange{P1:outer-octahedron}{P4:common-neighbors}.
For each vertex~$v \in V$ denote by~$SQ(v)$ the spherical quadrangulation induced by~$N(v)$ given in \cref{P3:quadrangulation}.
%We shall construct a generic boxed Plattenbau $\R$ whose touching graph is $G$.
% 
By \cref{P3:quadrangulation}, the out-neighbors of vertex~$v$ induce a 4-cycle in~$SQ(v)$, which we call the \term{equator}~$O_v$ of~$SQ(v)$.
The equator~$O_v$ splits the spherical quadrangulation~$SQ(v)$ into two \term{hemispheres}, each being a plane embedded quadrangulation with outer face~$O_v$ with the property that each vertex of~$SQ(v) - O_v$ is contained in exactly one hemisphere.
The vertices of~$O_v$ are the outer vertices of either hemisphere.
Note that one hemisphere (or even both) may be trivial, namely when the equator bounds a face of~$SQ(v)$.

We proceed with a number of claims.
 
\begin{myclaim}\label{claim:hemisphere}
 In each hemisphere, each inner vertex has exactly two outgoing edges and no outer vertex has an outgoing inner edge.
\end{myclaim}
\begin{claimproof}
  Let $u$ be any inner vertex of a hemisphere of $SQ(v)$.  We shall first show
  that $u$ has at least two out-neighbors in that hemisphere.  We have
  $u \in N(v) - O_v$, i.e., $u$ is a neighbor of $v$ but not an out-neighbor.
  Hence, the edge $uv$ is directed from $u$ to $v$ and $v$ lies on the equator
  $O_u$ of $SQ(u)$.  As such, $v$ has two neighbors $w_1,w_2$ on the equator
  of $O_u$, i.e., $w_1,w_2$ are out-neighbors of $u$. Also $w_1,w_2 \in SQ(v)$
  since they are neighbors of $v$. Hence, $u$ has at least two out-neighbors in
  $SQ(v)$ and, by planarity of $SQ(v)$, these are in the same hemisphere as
  $u$, as desired.
  
  Finally, each hemisphere of $v$ is a plane quadrangulation with outer
  4-cycle $O_v$, and as such has exactly $2k$ inner edges for $k$ inner
  vertices.  As each inner vertex has at least two outgoing edges, the edge
  count gives that each inner vertex has exactly two outgoing edges.
  Moreover, each inner edge is outgoing at some inner vertex, i.e., no outer
  vertex of the hemisphere has an outgoing inner edge.
\end{claimproof}

Together with~$v$ each equator edge of~$SQ(v)$ induces a triangle in~$G$.
These four triangles are the \term{equator triangles} of~$v$.
 
\begin{myclaim}\label{claim:equator-triangle}
 Every triangle in~$G$ is an equator triangle.
\end{myclaim}
\begin{claimproof}
  Let $\Delta$ be a triangle in $G$ with vertices $u,v,w$.  First, we shall
  show that~$\Delta$ is not oriented as a directed cycle, i.e., $\Delta$ has a
  vertex of out-degree two.  Without loss of generality let $uv$ be directed
  from $u$ to $v$.  If $uw$ is directed from $u$ to~$w$, we are done.  So
  assume that $uw$ is directed from $w$ to $u$.  Then $u \in N(v) - O_v$ is an
  inner vertex of a hemisphere of $SQ(v)$.  Moreover $w \in N(v)$ has an
  outgoing edge to $u$, and \cref{claim:hemisphere} implies that $w$ is an
  inner vertex of the same hemisphere of $SQ(v)$.  In particular, $vw$ is
  directed from $w$ to $v$ whence~$w$ is a vertex of out-degree two in $\Delta$.
  
 Now let $a$ be the vertex in $\Delta$ with out-degree two and $b,c$ be its out-neighbors in $\Delta$.
 Then~$b$ and $c$ lie on the equator of $a$ and are connected by an edge $bc$ in $G$.
 Thus $\Delta$ is an equator triangle.
\end{claimproof}

Clearly, a vertex~$w$ forms a triangle with two vertices~$u$ and~$v$ if and only if~$uv$ is an edge and~$w$ is a common neighbor of~$u$ and~$v$.
Equivalently,~$w$ is adjacent to~$v$ in~$SQ(u)$, which in turn is equivalent to~$w$ being adjacent to~$u$ in~$SQ(v)$.
Hence, the set~$N(u) \cap N(v)$ of all common neighbors (and thus also the set of all triangles sharing edge~$uv$) is endowed with the clockwise cyclic ordering around~$v$ in~$SQ(u)$, as well as with the clockwise cyclic ordering around~$u$ in~$SQ(v)$.
By \cref{P4:common-neighbors}, these two cyclic orderings are reversals of each other.
 
Let us define for a triangle~$\Delta$ in~$G$ with vertices~$u,v,w$ the two \term{sides} of~$\Delta$ as the two cyclic permutations of~$u,v,w$, which we denote by~\term{$[u,v,w]$} and~$[u,w,v]$.
So triangle~$\Delta$ has the two sides~$[u,v,w] = [v,w,u] = [w,u,v]$ and~$[u,w,v] = [w,v,u] = [v,u,w]$.
We define a binary relation~$\thicksim$ on the set of all sides of triangles in~$G$ as follows.
 \begin{equation}
  [u,v,a] \thicksim [v,u,b] \text{ if } 
   \begin{cases}
    \text{$a$ comes immediately before~$b$}\\
    \text{in the clockwise cyclic ordering}\\
    \text{of~$N(u) \cap N(v)$ around~$v$ in~$SQ(u)$}
   \end{cases}
   \label{eq:relation}
 \end{equation}
Note that by \ref{P4:common-neighbors}~$a$ comes immediately before~$b$ in the clockwise ordering around~$v$ if and only if~$b$ comes immediately before~$a$ in the clockwise ordering around~$u$.
Thus~$[u,v,a] \thicksim [v,u,b]$ also implies~$[v,u,b] \thicksim [u,v,a]$, i.e.,~$\thicksim$ is a symmetric relation and as such encodes an undirected graph~$H$ on the sides of triangles.

\begin{myclaim}\label{claim:cells}
 Each connected component of~$H$ is a cube.
 The corresponding subgraph in~$G$ is an octahedron.
\end{myclaim}
\begin{claimproof}
  Consider any fixed vertex $[u,v,w]$ of $H$.  Then $vw$ is an edge of $G$
  contained in $SQ(u)$.  As $SQ(u)$ is a quadrangulation, vertex $v$ has
  degree at least two in $SQ(u)$.  Hence, there exists a unique vertex $a$ in
  $SQ(u)$ such that $[u,v,w] \thicksim [v,u,a]$ according to
  \cref{eq:relation}.  Moreover, $a$ and $w$ are both neighbors of $v$ in
  $SQ(u)$ and hence non-adjacent in $G$, i.e.,
  $a \in (N(u) \cap N(v)) - N(w)$.  Symmetrically, we find
  $b \in (N(w) \cap N(u)) - N(v)$ with $[w,u,v] \thicksim [u,w,b]$ and
  $c \in (N(v) \cap N(w)) - N(u)$ with $[v,w,u] \thicksim [w,v,c]$.  It
  follows that $a,b,c$ are pairwise distinct vertices of $G$ and thus
  $[u,v,w]$ has degree exactly three in $H$.
  
  Now recall that $vw$ is an edge of $G$ contained in $SQ(u)$.  Consider the
  face $f$ in $SQ(u)$ for which $v$ comes immediately before $w$ in the
  clockwise ordering around $f$.  Let $v,w,s,t$ be the clockwise ordering of
  vertices around $f$.  Then, for example, $w$ comes immediately before $t$ in
  the clockwise cyclic ordering around $v$ in $SQ(u)$.  Using
  \cref{eq:relation}, we have the following.
 \[
  [u,v,w] \thicksim [v,u,t] = [u,t,v] \thicksim [t,u,s] = [u,s,t] \thicksim [s,u,w] = [u,w,s] \thicksim [w,u,v] = [u,v,w]
 \]
%
%   \begin{multline*}
%    [u,v,w] \thicksim [v,u,t] = [u,t,v] \thicksim [t,u,s]\\ = [u,s,t] \thicksim [s,u,w] = [u,w,s] \thicksim [w,u,v] = [u,v,w]
%   \end{multline*}
%   
  It follows that $t=a$ and $s=b$.
  I.e., the component of $H$ with vertex $[u,v,w]$ contains the four triangle sides $[u,v,w]$, $[u,w,b]$, $[u,b,a]$, $[u,a,v]$ and these form a 4-cycle in $H$.
  As $v,w,a,b$ are pairwise distinct vertices in $G$, we have $[u,v,w] \not\thicksim [u,b,a]$ and $[u,w,b] \not\thicksim [u,a,v]$, meaning that the above 4-cycle in $H$ is induced.
  
  \begin{figure}
   \centering
   \includegraphics{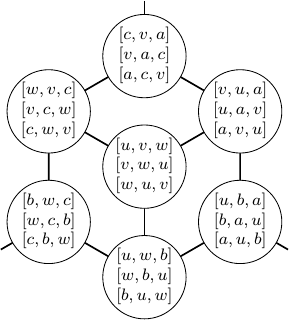}
   \caption{The second neighborhood of vertex $[u,v,w]$ in the auxiliary graph $H$.}
   \label{fig:cube-component}
  \end{figure}
  
  Repeating the same argument for $[v,w,u]$ in quadrangulation $SQ(v)$ and $[w,u,v]$ in quadrangulation $SQ(w)$, we get the induced 4-cycles
   \[
    [v,w,u] \thicksim [w,v,c] = [v,c,w] \thicksim [c,v,a] = [v,a,c] \thicksim [a,v,u] = [v,u,a] \thicksim [u,v,w] = [v,w,u]
   \] 
% 
%   \begin{multline*}
%    [v,w,u] \thicksim [w,v,c] = [v,c,w] \thicksim [c,v,a]\\ = [v,a,c] \thicksim [a,v,u] = [v,u,a] \thicksim [u,v,w] = [v,w,u]
%   \end{multline*}
% 
  and
%   
%  \[
%   [w,u,v] \thicksim [u,w,b] = [w,b,u] \thicksim [b,w,c] = [w,c,b] \thicksim [c,w,v] = [w,v,c] \thicksim [v,w,u] = [w,u,v].
%  \]
% 
  \begin{multline*}
   [w,u,v] \thicksim [u,w,b] = [w,b,u] \thicksim [b,w,c]\\ = [w,c,b] \thicksim [c,w,v] = [w,v,c] \thicksim [v,w,u] = [w,u,v].
  \end{multline*}
  As $[u,v,w] = [v,w,u] = [w,u,v]$, these three induced 4-cycles in $H$ pairwise share exactly one edge, as shown in \cref{fig:cube-component}.
  For each of the three vertices $[a,c,v]$, $[b,a,u]$, and $[c,b,w]$, the third neighbor in $H$ is yet to be explored.
  However by symmetry, each of those vertices is also in exactly three induced 4-cycles, which implies that they have the same third neighbor: vertex $[a,b,c] = [b,c,a] = [c,a,b]$.
  
  Thus, the component of $H$ containing $[u,v,w]$ is a cube.
  The eight corresponding triangles in $G$ form an octahedron with vertex set $\{u,v,w,a,b,c\}$.
 \end{claimproof}

 With \cref{claim:cells} we have identified a family of octahedra in~$G$ such
 that each side of each triangle in~$G$ is contained in exactly one
 octahedron.  We call these octahedra the \term{cells} of~$G$, as these
 correspond in the 2-dimensional case to the 4-cycles bounding faces of
 the quadrangulation.  As \cref{eq:relation} puts two triangle sides~$[u,v,a]$
 and~$[v,u,b]$ into a common cell if and only if the edges~$va$ and~$vb$ bound
 the same facial 4-cycle in~$SQ(u)$, we have the following correspondence
 between the cells of~$G$ and the faces in the spherical quadrangulations.
 
\begin{myclaim}\label{claim:face-in-cell}
  If~$O \subseteq G$ is a cell of~$G$ and~$C$ is an induced 4-cycle in~$O$,
  then~$C$ bounds a face of~$SQ(v)$ and a face of~$SQ(u)$ for the two
  vertices~$u,v \in O - C$.  Conversely, if~$C$ is a 4-cycle bounding a face
  of~$SQ(v)$, then there is a cell~$O$ of~$G$ containing~$\{v\} \cup V(C)$.
\end{myclaim}
 
Having identified the cells, we can now construct a generic boxed Plattenbau
for~$G$ by identifying two opposite vertices in a particular cell, calling
induction, and then splitting the rectangle corresponding to the
identification vertex into two.  The cells of~$G$ will then correspond to the
rooms in~$\R$, except that one cell of~$G$ will correspond to the unbounded
region of~$\R$ (which is not a room).  To this end, we prove the following
stronger statement:

\begin{lemma}\label{lem:induction}
 Let~$G$ be a graph satisfying \crefrange{P1:outer-octahedron}{P4:common-neighbors} and let~$A,B,C$ be three outer vertices forming a triangle in~$G$.
 Then there exists a generic boxed Plattenbau~$\R$ whose touching graph is~$G$ such that each of the following holds.
 \begin{enumerate}[leftmargin=2.1em,label=(I\arabic*)]
  \item The six outer vertices of~$G$ correspond to the outer rectangles of~$\R$.\label{item:induction-1}
  \item The cells of~$G$ correspond to the rooms of~$\R$, except for one cell that is formed by all six outer vertices.\label{item:induction-2}
  \item For any two vertices~$u,v$ with corresponding rectangles~$R_u,R_v$ we have~$u \to v$ in the orientation of~$G$ if and only if~$R_u \cap R_v$ contains an edge of~$R_u$.\label{item:induction-3}
  \item For each vertex~$v$ corresponding to a rectangle~$R_v$, the rectangles touching~$R_v$ come in the same spherical order as their corresponding vertices in~$SQ(v)$.\label{item:induction-4}
 \end{enumerate}
\end{lemma}
\begin{proof}
  We proceed by induction on the number of vertices in $G = (V,E)$.
  As the base case we have $|V| = 6$ and $G \cong K_{2,2,2}$ is just an octahedron.
  In this case $G$ has exactly eight triangles and 16 sides of triangles.
  There are exactly two cells, each isomorphic to $G$.
  The desired Plattenbau $\R$ is given by the six sides of an axis-aligned cuboid in $\RR^3$.
  It is easy to see that $\R$ has the required properties.
  
  So let us assume that $|V| > 6$, i.e., there is at least one inner vertex.
  Consider the three outer vertices $A,B,C$, which form a triangle in $G$.
  We shall first show that at least one of the quadrangulations $SQ(A)$, $SQ(B)$, $SQ(C)$ has an inner vertex.
  As $G$ is connected by \cref{P1:outer-octahedron}, there is an edge in $G$ from some inner vertex $v$ to some outer vertex $w$ with $w$ not necessarily in $\{A,B,C\}$.
%  If $w \in \{A,B,C\}$, we are done.
  Then $v$ is an inner vertex of $SQ(w)$, i.e., one hemisphere of $SQ(w)$ is a plane quadrangulation with at least five vertices whose outer 4-cycle consists of four outer vertices of $G$.
  One edge of the outer 4-cycle has both endpoints in $\{A,B,C\}$ and, as there is some inner vertex, at least one of these two endpoints has an inner vertex as a neighbor.
  
  So we may assume that $A$ has an inner neighbor, i.e., the non-trivial hemisphere $Q$ of $SQ(A)$ has an inner face $f$ that is bounded by the edge $BC$ and at least one inner vertex of $G$.
  Consider the cell $O$ of $G$ that contains all vertices on $f$, as given by \cref{claim:face-in-cell}.
  Let the vertices of this octahedron $O$ be denoted by $A,B,C,a,b,c$ with the three pairs of non-adjacent vertices being $\{A,a\}$, $\{B,b\}$, and $\{C,c\}$.
  Note that at least one of $a,b,c$ is in inner vertex of $G$ since $O$ includes an inner vertex of $Q$.
  In any case, vertices $a,b,c$ form a triangle in $G$ and each of the three vertices has outgoing edges to two of $A,B,C$.
  By \cref{claim:equator-triangle}, one of $a,b,c$, say $c$, has outgoing edges to the other two vertices ($a$ and $b$ in this case).
  In particular, $c$ is an inner vertex, as otherwise $c$ has only outer vertices as out-neighbors (by \cref{P2:4-orientation}), contradicting that one of $a,b,c$ is an inner vertex.
  
  Now we identify vertices $c$ and $C$ in $G$, denoting the resulting vertex by $\tilde{C}$ and the resulting graph by $\tilde{G}$.
  Each of $A,B,a,b$ is a neighbor of $c$ and $C$ in $G$.
  We remove double edges during the identification, so that in $\tilde{G}$ vertex $\tilde{C}$ is connected to each of $A,B,a,b$ with a single edge.
  In $\tilde{G}$ we choose the same six vertices except for $\tilde{C}$ replacing $C$ as the outer vertices.
  We claim that $\tilde{G}$ together with this choice of outer vertices has the properties in \crefrange{P1:outer-octahedron}{P4:common-neighbors}.
  Indeed \cref{P1:outer-octahedron} holds because each neighbor $w$ of $\tilde{C}$ in $\tilde{G}$ that is not a neighbor of $C$ in $G$ has in $G$ an incoming edge to $c$.
  In particular, such $w$ is an inner vertex, as $c$ is an inner vertex.
  Hence the outer vertices in $\tilde{G}$ induce an octahedron.
  Connectivity of $\tilde{G}$ follows from connectivity of $G$.
  
  For \cref{P2:4-orientation} we first prove the following.
  
  \begin{myclaim}\label{claim:common-neighbors-of-C}
   We have $N(c) \cap N(C) = \{A,B,a,b\}$.
  \end{myclaim}
  \begin{claimproof}
    Suppose for the sake of contradiction that $v$ is a common neighbor of $c$
    and~$C$ different from $A,B,a,b$.  Then~$C$ and~$c$ are out-neighbors of
    $v$, i.e., $v$ is a non-equator vertex in~$SQ(c)$ and $SQ(C)$.  Let $w$ be
    an out-neighbor of $v$ different from $c,C$.  Then, by
    \cref{claim:hemisphere}~$w$ is contained in $SQ(c)$ and $SQ(C)$, too.  If
    $w$ is neither on the equator of $SQ(c)$ nor on the equator of $SQ(C)$, we
    can repeat the argument with $w$ taking the role of $v$.  Thus assume
    that~$w$ is an out-neighbor of at least one of $c,C$.  As out-neighbors of
    the outer vertex $C$ are outer vertices, it follows that $w$ is an
    out-neighbor of $c$, i.e., $w \in \{A,B,a,b\}$.  By symmetry, assume that
    $w \in \{B,b\}$.  Now consider the hemisphere $Q$ of $SQ(w)$ that contains
    $v$ as an inner vertex.  As $A,B,C,a,b,c$ form a cell, the vertices
    $A,C,a,c$ form a quadrangular face in $Q$ by \cref{claim:face-in-cell}.
    Moreover, we know that $c$ has outgoing edges to $A$ and $a$, while $v$
    has outgoing edges to $c$ and~$C$, see \cref{fig:claim-neighbors-of-C} for
    an illustration of the situation on $SQ(w)$.
     
   \begin{figure}
    \centering
    \includegraphics{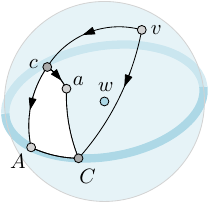}
    \caption{The situation in \cref{claim:common-neighbors-of-C} that leads to a contradiction.}
    \label{fig:claim-neighbors-of-C}
   \end{figure}

   Vertices $a,A,c,C$ and $v$ induce a $K_{2,3}$ in $SQ(w)$.  Since $a,A,c,C$
   bound a face of the hemisphere $Q$ of $SQ(w)$, and $A,C$ are outer vertices
   of $G$ (hence outer vertices of $Q$), it follows that vertex $a$ lies
   inside the 4-cycle $K$ formed by $A,c,v,C$ in $Q$.  As $Q$ is a
   quadrangulation, the 4-cycle $K$ together with the vertices in its
   interior is as well a quadrangulation $J$ with outer face $K$.  Hence,
   $J$ has  $|V(J)|-4$ inner vertices and $2|V(J)|-8$ inner edges.  One
   of the inner edges of $J$, namely the edge $ca$ is outgoing at the outer
   vertex $c$ of $J$.  Hence at most $2|V(J)|-9$ inner edges of $J$ are
   outgoing at an inner vertex of $J$.  But this is a contradiction to
   \cref{claim:hemisphere}, which states that each of the $|V(J)|-4$ inner
   vertices has exactly two outgoing edges in $Q$.
  \end{claimproof}

  \cref{claim:common-neighbors-of-C} implies that the orientation of $G$ given by \cref{P2:4-orientation} naturally induces an orientation of $\tilde{G}$ satisfying again \cref{P2:4-orientation}.

  \begin{myclaim}\label{claim:P3-holds}
   \cref{P3:quadrangulation} holds for $\tilde{G}$.
  \end{myclaim}
  \begin{claimproof}
    We shall show that for each vertex $v$ of $\tilde{G}$, the neighborhood of
    $v$ induces a spherical quadrangulation $SQ(v)$ in which the out-neighbors
    of $v$ form a 4-cycle.  We distinguish different cases of $v$ and how
    $N(v)$ changes during the identification of $c$ and $C$.  For
    $v = \tilde{C}$, $N(v)$ in $\tilde{G}$ is the union of $N(c)$ and $N(C)$
    in $G$.  As $A,B,C,a,b,c$ form a cell, $A,B,a,b$ form a face in both
    $SQ(c)$ and $SQ(C)$ by \cref{claim:face-in-cell}.  Moreover, $A,B,a,b$ are
    the equator in $SQ(c)$.  Hence, the subgraph of $\tilde{G}$ induced by
    $N(v) = N(c) \cup N(C)$ can be obtained by pasting $SQ(c)$ into the face
    $A,B,a,b$ of $SQ(C)$.  This is a quadrangulation and the out-neighbors of
    $\tilde{C}$ are the same as for $C$, i.e., induce a 4-cycle.  For
    $v \in \{A,B,a,b\}$, identifying $C$ and $c$ corresponds to merging
    opposite vertices of a face $f$ in $SQ(v)$.  As
    $N(c) \cap N(C) = \{A,B,a,b\}$ by \cref{claim:common-neighbors-of-C}, in
    $SQ(v)$ vertices $c$ and $C$ have no common neighbor outside of $f$.  Thus
    the identification of~$C$ and~$c$ in $SQ(v)$ preserves the property of
    being a quadrangulation and does not affect the equator.  For
    $v \notin \{\tilde{C},A,B,a,b\}$ the neighborhood of $v$ does not change,
    except for possibly renaming $c$ or $C$ to $\tilde{C}$.  This shows that
    $\tilde{G}$ satisfies \cref{P3:quadrangulation}.
  \end{claimproof}

  \begin{myclaim}\label{claim:P4-holds}
   \cref{P4:common-neighbors} holds for $\tilde{G}$.
  \end{myclaim}
  \begin{claimproof}
   Recall that during the identification of $c$ and $C$, we changed the embedding of $SQ(v)$ only for $v \in \{\tilde{C},A,B,a,b\}$.
   Thus we need to check only those edges in $\tilde{G}$ with at least one endpoint in $\{\tilde{C},A,B,a,b\}$.
   First consider an edge $\tilde{C}v$ with $v \notin \{A,B,a,b\}$.
%  Let $v\notin\{B,C,b,c\}$ be a neighbor of $\tilde{A}$.
   If $v$ was a neighbor of $c$ in $G$, then in $SQ(\tilde{C})$ vertex $v$ is embedded inside the quadrangle $A,B,a,b$.
   If $v$ was a neighbor of $C$ in $G$, then in $SQ(\tilde{C})$ vertex $v$ is embedded outside the quadrangle $A,B,a,b$.
   By planarity of $SQ(\tilde{C})$, there is no edge between these two types of vertices.
   Thus the common neighborhood of $\tilde{C}$ and any $v \notin \{A,B,a,b\}$ coincides with either $N(c) \cap N(v)$ or $N(C) \cap N(v)$ in $G$.
   Say $v$ was a neighbor of $c$.
   As the clockwise cyclic ordering around $v$ in $SQ(\tilde{C})$ is the same as in $SQ(c)$, and the clockwise ordering around $\tilde{C}$ in $SQ(v)$ is the same as that around $c$ in $SQ(v)$ before, \cref{P4:common-neighbors} is satisfied here.
   The case of an edge $\tilde{C}v$ with $v \in \{A,B,a,b\}$ and the other cases are similar.
    \end{claimproof}

    Up to now we have shown that the graph $\tilde{G}$ obtained from $G$ by
    identifying $c$ and $C$ satisfies
    \crefrange{P1:outer-octahedron}{P4:common-neighbors}.  Thus, by induction
    there is a generic boxed Plattenbau $\R'$ whose touching graph is
    $\tilde{G}$ such that \crefrange{item:induction-1}{item:induction-4} hold.
    In particular the rectangles $R_A,R_B,R_a,R_b$ for the 4-cycle $A,B,a,b$
    in $SQ(\tilde{C})$ enclose a rectangular region in one corner of the
    rectangle $R$ for~$\tilde{C}$ (possibly entire $R$).  We alter $\R$ by
    shortening all rectangles touching $R$ inside this region by the same
    small amount~$\varepsilon >0$ and introducing a new rectangle $R_c$ for
    $c$ parallel to $R$ at distance~$\varepsilon$, touching all shortened
    rectangles and the rectangles $R_A,R_B,R_a,R_b$.  Secondly, we let $R$ be
    the rectangle for $C$.  Then the resulting Plattenbau $\R$ represents $G$
    as its touching graph, and \crefrange{item:induction-1}{item:induction-4}
    hold for $\R$.
 \end{proof}

\cref{lem:induction} shows the sufficiency of \crefrange{P1:outer-octahedron}{P4:common-neighbors}.
The necessity is given in \cref{prop:boxed-necessity}.
Together this proves \cref{thm:proper-boxed}.

\subsection{Iterative Constructions for generic boxed Plattenbauten}\label{subsec:iterative}

We have given a characterization of graphs of generic boxed Plattenbauten as a generalization of plane quadrangulation. In this section, we show, how these Plattenbauten can be constructed iteratively by inserting Plattenbauten into each other starting from the trivial one.

The following propositions give two different descriptions on the iterative structure of generic boxed Plattenbauten.

\begin{proposition}\label{prop:corner-reduction}
  If $\R$ is a generic boxed Plattenbau with at least two rooms and
  $Z$ is a room containing an outer corner, then $Z$ has a side $A_0$, that  is a rectangle of $\R$.
\end{proposition}

\begin{proof}
  Let $o$ be an outer corner and $Z$ be the room of $\R$ which contains $o$.
  Let $\bar{o}$ be the corner of $Z$ which is opposite to $o$.  Let
  $A_0,A_1,A_2$ be the three rectangles of $\R$ which form the sides of $Z$
  containing $\bar{o}$.  Let $A_0$ be the one which has $\bar{o}$ as a corner,
  $A_1$ be the one which has $\bar{o}$ on a boundary edge, and $A_2$ be the
  one which has $\bar{o}$ as an interior point.  This local structure
  at~$\bar{o}$ shows that $A_0\to A_1$, $A_0\to A_2$, and $A_1\to A_2$.  The
  first two of these imply that $A_0$ is a rectangle of $\R$.
\end{proof}

\cref{prop:corner-reduction} shows that a generic boxed Plattenbau $\R$ with
outer corner $o$ can be reduced to a trivial Plattenbau with only one room by
repeating the following step: Identify the rectangle~$A_0$, remove it and
extend all the rectangles $B$ with $B\to A_0$ such that the edge of~$B$ which
made a contact with $A_0$ makes a contact with the outer rectangle which
contains~$o$ and is parallel to $A_0$.  This reduction also follows from
the inductive proof of \cref{thm:proper-boxed}.  A consequence that can be
drawn from the reduction is that the number of combinatorially different
generic boxed Plattenbauten with $n$ rectangles is only exponential in $n$.  A
rough estimate gives an upper bound of $24^n$.

The next proposition gives a decomposition of generic boxed Plattenbauten by
means of ``stacking'' Plattenbauten into Plattenbauten.

\begin{proposition}
  Let $\R$ be a generic boxed Plattenbau with inner rectangles in all
  three directions and $r$ rooms. Then there are generic boxed
  Plattenbauten $\R_O$ and $\R_I$ with $r_O >1$ and $r_I> 1$ rooms
  respectively and a room $Z$ of $\R_O$, such that
  $\R$ can be obtained by inserting the inner rectangles of $\R_I$ in
  the room $Z$ of $\R_O$, i.e., $r = r_o+r_I-1$.
\end{proposition}

\begin{proof}
  Let $x,y,z$ be the axes of the coordinate system. By translating and scaling
  the Plattenbau we may assume that the outer box of the Plattenbau is the
  cube spanned by $(0,0,0)$ and $(1,1,1)$. The outer rectangle in the $z=0$
  plane is the \emph{bottom rectangle} $R_0$ of $\R$. The \emph{top rectangle}
  is the rectangle in a plane $z=1$, we denote it $R_1$. All the contacts of
  $R_0$ with inner rectangles are of type $A\to R_0$, i.e., $A \cap R_0$ is a
  segment on $R_0$.

  If there is no inner rectangle $A$ with $A\to R_0$, then there is a
  unique room $Z$ with a side on~$R_0$. Let $B$ be the side of $Z$
  opposite of the $R_0$-side. Now let $R_I = \R - R_0$ and let $R_O$
  consist of the outer rectangles of $R$ together with $B$, i.e.,
  $r_O=2$ and $r_I = r-1$.
  
  From now on we assume that there is an inner rectangle $A$ with
  $A\to R_0$, i.e., a segment in $R_0$. The union of all these segments
  yields a rectangular dissection $D$ of the unit square $U\subseteq
  R_0$. With each inner segment $s$ of $D$ there is a rectangle
  $R_s$. Let $z_s$ be the maximum $z$ coordinate of a point in $R_s$, we
  refer to $z_s$ as the height of $R_s$.

  If $s$ and $s'$ are segments in $R_0$ and $R_s \to R_{s'}$, then
  $z_s \leq z_{s'}$ because the representation is proper. This shows that if
  we fix some $h$ with $0\leq h\leq 1$ and only look at segments $s\in D$ with
  $z_s > h$ we get a dissection $D_h$ of $U$ into rectangles.
  Let $h^+$ be the maximum value of $z_s$ taken over inner segments $s\in D$.

  If $h^+=1$ we use that the dissection
  $D_1$ is nontrivial and conclude that at least one of the rectangles of $D_1$ spans a box $B$
  between $R_0$ and $R_1$ which contains a rectangle $A$ in a plane $z=h$ with
  $0< h < a$. Let $\R_B$ be the set of all rectangles of
  $\R$ which are in $B$. Since $A\in \R_B$ this set is nonempty.
  Let $\R_I$ be $\R_B$ together with 6 outer rectangles covering
  the sides of $B$ and let $R_O = \R - \R_B$. This is a decomposition
  as claimed.
  
  If $h^+ < 1$ let $D_+$ be the dissection $D_{h^+}$. If $D_+$
  contains a guillotine segment, i.e., a segment spanned between opposite
  sides of the outer square of $D_+$. Then we can permute the
  coordinates such that after the permutation we have $h^+=1$, i.e.,
  we are in the previous case and have a nontrivial decomposition.

  If $h^+ < 1$ and there is no guillotine segment, then the
  contact system of inner segments of $D_+$ is connected
  and for each of the outer segments there is an inner segment
  having a contact with it. Now consider two segments $s$ and $s'$ in
  $D_+$ such that $R_s \to R_{s'}$. Since $\R$ is generic there is some rectangle
  $T_{s'}$ containing the (open) top segment of  $R_{s'}$ in the interior, i.e,
  $T_{s'}$ contains some $\varepsilon$ stripe on both sides of the intersection
  of $R_{s'}$ with the plane $z=h^+$. Similarly there is a rectangle $T_s$
  containing the  top segment of  $R_{s}$. Hence, $T_{s'}$ and $T_{s}$ intersect
  whence  $T_{s'}=T_{s}$. Iterating this argument through all contacts of two segments of
  $D^+$ we find that there is a unique rectangle $T$ such that $R_s \to T$ for all
  interior segments $s$ of $D_+$. This shows that $T$ spans the outer square of $D_+$.
  Now let $\R_B$ be the set of all rectangles of
  $\R$ which have $z$-coordinates between $0$ and $h^+$, except $T$.
  This set is nonempty.
  Let $\R_I$ be $\R_B$ together with~6 outer rectangles covering
  the sides of $B$ and let $R_O = \R - \R_B$. This is a decomposition
  as claimed.
\end{proof}

%%%%%%%%%%%%%%%%%%%%%%%%%%%%%
%%                         %%
%%    SECTION Conclusion   %%
%%                         %%
%%%%%%%%%%%%%%%%%%%%%%%%%%%%%
\section{Conclusions}\label{sec:Conclusions}

We have studied touching graphs of (generic) Plattenbauten as generalizations of planar bipartite graphs to space. Our main results are that 3-chromatic planar graphs belong to this class, and the characterization of touching graphs of generic boxed Plattenbauten as a generalization of planar quadrangulations. However, a full characterization of touching graphs of (generic) Plattenbauten remains challenging. 

With our results at hand it is natural to try and extend results from the planar setting to space.
One example, is to attack a question asked by Jean Cardinal at the Order \& Geometry Workshop at Gułtowy Palace in 2016:
What is the 3-dimensional analogue of Baxter permutations? 
Since Baxter permutations are in bijection with boxed arrangements of axis-parallel segments in~$\RR^2$~\cite{ffno-bbfro-11}, the question aims at finding permutation-like objects corresponding to generic boxed Plattenbauten. Our iterative constructions from Subsection~\ref{subsec:iterative} might help.

A natural continuation of this project is going to higher dimensions, i.e., consider touching graphs of cuboids of co-dimension one  in~$\mathbb{R}^d$. Already the class of 4-dimensional Plattenbau graphs contains all planar graphs, which follows from~\cite{thomassen1986interval}. However, our constructions for planar 3-chromatic graphs could be generalized to higher dimensions, i.e., orthogonal subspaces in~$\mathbb{R}^d$. A first interesting question here would be to characterize   the skeleta of orthogonal subspaces in~$\mathbb{R}^d$. 

Finally, as suggested in the introduction, considering the intersection graph~$I_{\R}$ instead of the touching graph of a Plattenbau, yields an interesting but very different graph class. The plane analogue of this is known as B$_0$-CPG graphs~\cite{CPG18}, yielding a first subclass.
Another subclass are 4-connected planar graphs, since they have a rectangle contact representation in~$\RR^2$, see~\cite{ung-53,leinwand1984algorithm,kozminski1985rectangular,thomassen1986interval,rosenstiehl1986rectilinear}.
Also~$K_{12}$ is the intersection graph of the Plattenbau~$\R$ consisting of the twelve axis-parallel unit squares in~$\RR^3$ that have a corner on the origin.

% If point-intersections are disregarded, still $K_6$ can be represented in a similar fashion using rectangles of size $1\times 2$.

%%%%%%%%%%%%%%%%%%%%%%%%%%%%%%%%%%%%%%%%%%%%%%%%%%%%%%%%%%%%%%%%%%%%%%%%%
% \def\SC{\scshape}
% \let\sc=\SC
% \bibliographystyle{my-siam}

\def\path{} %% this is need as \path is redefined in the included pdffig.sty
\bibliographystyle{plainurl}

\bibliography{lit}

\begin{thebibliography}{10}

\bibitem{bowers2010circle}
Philip~L. Bowers.
\newblock Circle packing: a personal reminiscence.
\newblock In Mircea Pitici, editor, {\em The Best Writing on Mathematics 2010},
  pages 330--345. Princeton University Press, 2010.

\bibitem{buchsbaum2008rectangular}
Adam~L. Buchsbaum, Emden~R. Gansner, Cecilia~M. Procopiuc, and Suresh
  Venkatasubramanian.
\newblock Rectangular layouts and contact graphs.
\newblock {\em ACM Transactions on Algorithms}, 4(1), 2008.
\newblock \href {https://doi.org/10.1145/1328911.1328919}
  {\path{doi:10.1145/1328911.1328919}}.

\bibitem{Cha-11}
{L. Sunil} Chandran, Rogers Mathew, and Naveen Sivadasan.
\newblock Boxicity of line graphs.
\newblock {\em Discrete Mathematics}, 311(21):2359--2367, 2011.
\newblock \href {https://doi.org/10.1016/j.disc.2011.06.005}
  {\path{doi:10.1016/j.disc.2011.06.005}}.

\bibitem{de2001topological}
Hubert De~Fraysseix and Patrice Ossona~de Mendez.
\newblock On topological aspects of orientations.
\newblock {\em Discrete Mathematics}, 229(1):57--72, 2001.
\newblock \href {https://doi.org/10.1016/S0012-365X(00)00201-6}
  {\path{doi:10.1016/S0012-365X(00)00201-6}}.

\bibitem{CPG18}
Zakir Deniz, Esther Galby, Andrea Munaro, and Bernard Ries.
\newblock On contact graphs of paths on a grid.
\newblock In {\em Graph Drawing and Network Visualization}, volume 11282 of
  {\em LNCS}, pages 317--330. Springer, 2018.
\newblock \href {https://doi.org/10.1007/978-3-030-04414-5\_22}
  {\path{doi:10.1007/978-3-030-04414-5\_22}}.

\bibitem{EppM-14}
David Eppstein and Elena Mumford.
\newblock Steinitz theorems for simple orthogonal polyhedra.
\newblock {\em Journal of Computational Geometry}, 5(1):179--244, 2014.
\newblock \href {https://doi.org/10.20382/jocg.v5i1a10}
  {\path{doi:10.20382/jocg.v5i1a10}}.

\bibitem{f-cdpgo-01}
Stefan Felsner.
\newblock Convex drawings of planar graphs and the order dimension of
  3-polytopes.
\newblock {\em Order}, 18:19--37, 2001.
\newblock \href {https://doi.org/10.1023/A:1010604726900}
  {\path{doi:10.1023/A:1010604726900}}.

\bibitem{f-gga-04}
Stefan Felsner.
\newblock {\em Geometric Graphs and Arrangements}.
\newblock Vieweg Verlag, 2004.
\newblock \href {https://doi.org/10.1007/978-3-322-80303-0}
  {\path{doi:10.1007/978-3-322-80303-0}}.

\bibitem{f-lspg-04}
Stefan Felsner.
\newblock Lattice structures from planar graphs.
\newblock {\em The Electronic Journal of Combinatorics}, 11(R15):24p., 2004.
\newblock \href {https://doi.org/10.37236/1768} {\path{doi:10.37236/1768}}.

\bibitem{felsner2013rectangle}
Stefan Felsner.
\newblock Rectangle and square representations of planar graphs.
\newblock In J.~Pach, editor, {\em Thirty Essays on Geometric Graph Theory},
  pages 213--248. Springer, 2013.
\newblock \href {https://doi.org/10.1007/978-1-4614-0110-0\_12}
  {\path{doi:10.1007/978-1-4614-0110-0\_12}}.

\bibitem{f-odpmr-14}
Stefan Felsner.
\newblock The order dimension of planar maps revisited.
\newblock {\em SIAM Journal on Discrete Mathematics}, 28:1093--1101, 2014.
\newblock \href {https://doi.org/10.1137/130945284}
  {\path{doi:10.1137/130945284}}.

\bibitem{ffno-bbfro-11}
Stefan {Felsner}, \'Eric {Fusy}, Marc {Noy}, and David {Orden}.
\newblock {Bijections for Baxter families and related objects}.
\newblock {\em Journal of Combinatorial Theory, Series A}, 118(3):993--1020,
  2011.
\newblock \href {https://doi.org/10.1016/j.jcta.2010.03.017}
  {\path{doi:10.1016/j.jcta.2010.03.017}}.

\bibitem{WG-Plattenbauten}
Stefan Felsner, Kolja Knauer, and Torsten Ueckerdt.
\newblock Plattenbauten: {T}ouching rectangles in space.
\newblock In {\em Graph-Theoretic Concepts in Computer Science}, volume 12301
  of {\em LNCS}, pages 161--173. Springer, 2020.
\newblock \href {https://doi.org/10.1007/978-3-030-60440-0\_13}
  {\path{doi:10.1007/978-3-030-60440-0\_13}}.

\bibitem{fz-os3d-06}
Stefan Felsner and Florian Zickfeld.
\newblock Schnyder woods and orthogonal surfaces.
\newblock {\em Discrete and Computational Geometry}, 40:103--126, 2008.
\newblock \href {https://doi.org/10.1007/s00454-007-9027-9}
  {\path{doi:10.1007/s00454-007-9027-9}}.

\bibitem{fusy2007combinatoire}
{\'E}ric Fusy.
\newblock {\em Combinatoire des cartes planaires et applications
  algorithmiques}.
\newblock PhD thesis, LIX Polytechnique, 2007.
\newblock URL:
  \url{http://www.lix.polytechnique.fr/Labo/Eric.Fusy/Theses/these_eric_fusy.pdf}.

\bibitem{g-3cpghaicru3s-19}
Daniel Gon\c{c}alves.
\newblock 3-colorable planar graphs have an intersection segment representation
  using 3 slopes.
\newblock In {\em Graph-Theoretic Concepts in Computer Science}, volume 11789
  of {\em LNCS}, pages 351--363. Springer International Publishing, 2019.
\newblock \href {https://doi.org/10.1007/978-3-030-30786-8\_27}
  {\path{doi:10.1007/978-3-030-30786-8\_27}}.

\bibitem{s-cee-73}
Terja Hansen and Herbert Scarf.
\newblock {\em The Computation of Economic Equilibria}, volume~24 of {\em
  Cowles Foundation Monograph}.
\newblock Yale Univ. Press, 1973.
\newblock URL: \url{https://lccn.loc.gov/73077165}.

\bibitem{hartman1991grid}
I.Ben-Arroyo Hartman, Ilan Newman, and Ran Ziv.
\newblock On grid intersection graphs.
\newblock {\em Discrete Mathematics}, 87(1):41--52, 1991.
\newblock \href {https://doi.org/10.1016/0012-365X(91)90069-E}
  {\path{doi:10.1016/0012-365X(91)90069-E}}.

\bibitem{hm-odcg-99}
Serkan Ho\c{s}ten and Walter~D. Morris.
\newblock The order dimension of the complete graph.
\newblock {\em Discrete Mathematics}, 201(1):133--139, 1999.
\newblock \href {https://doi.org/10.1016/S0012-365X(98)00315-X}
  {\path{doi:10.1016/S0012-365X(98)00315-X}}.

\bibitem{koebe1936kontaktprobleme}
Paul Koebe.
\newblock {Kontaktprobleme der konformen Abbildung}.
\newblock {\em Berichte \"uber die Verhandlungen der S\"achsischen Akademie der
  Wissenschaften zu Leipzig, Mathematisch-Physische Klasse}, 88:141--164, 1936.

\bibitem{kozminski1985rectangular}
Krzysztof Ko{\'z}mi{\'n}ski and Edwin Kinnen.
\newblock Rectangular duals of planar graphs.
\newblock {\em Networks}, 15(2):145--157, 1985.
\newblock \href {https://doi.org/10.1002/net.3230150202}
  {\path{doi:10.1002/net.3230150202}}.

\bibitem{leinwand1984algorithm}
Sany~M Leinwand and Yen-Tai Lai.
\newblock An algorithm for building rectangular floor-plans.
\newblock In {\em 21st Design Automation Conference Proceedings}, pages
  663--664. IEEE, 1984.
\newblock \href {https://doi.org/10.1109/DAC.1984.1585874}
  {\path{doi:10.1109/DAC.1984.1585874}}.

\bibitem{m-pgmrtmi-02}
Ezra Miller.
\newblock Planar graphs as minimal resolutions of trivariate monomial ideals.
\newblock {\em Documenta Mathematica}, 7:43--90, 2002.
\newblock \href {https://doi.org/10.4171/DM/117} {\path{doi:10.4171/DM/117}}.

\bibitem{ms-cca-04}
Ezra Miller and Bernd Sturmfels.
\newblock {\em Combinatorial Commutative Algebra}, volume 227 of {\em Graduate
  Texts in Mathematics}.
\newblock Springer, 2004.
\newblock \href {https://doi.org/10.1007/b138602} {\path{doi:10.1007/b138602}}.

\bibitem{pach1994representation}
J{\'a}nos Pach, Hubert de~Fraysseix, and Patrice Ossona~de Mendez.
\newblock Representation of planar graphs by segments.
\newblock In K.~Böröczky and G.~Fejes Tóth, editors, {\em Intuitive
  Geometry}, Coll. Math. Soc. J. Bolyai 63, pages 109--117. North-Holland,
  1994.
\newblock URL:
  \url{https://infoscience.epfl.ch/record/129343/files/segments.pdf}.

\bibitem{rosenstiehl1986rectilinear}
Pierre Rosenstiehl and Robert~E Tarjan.
\newblock Rectilinear planar layouts and bipolar orientations of planar graphs.
\newblock {\em Discrete and Computational Geometry}, 1:343--353, 1986.
\newblock \href {https://doi.org/10.1007/BF02187706}
  {\path{doi:10.1007/BF02187706}}.

\bibitem{stephenson2005introduction}
Kenneth Stephenson.
\newblock {\em Introduction to circle packing: The theory of discrete analytic
  functions}.
\newblock Cambridge University Press, 2005.
\newblock \href {https://doi.org/10.1017/S0025557200180726}
  {\path{doi:10.1017/S0025557200180726}}.

\bibitem{tamassia1986unified}
Roberto Tamassia and Ioannis~G Tollis.
\newblock A unified approach to visibility representations of planar graphs.
\newblock {\em Discrete and Computational Geometry}, 1:321--341, 1986.
\newblock \href {https://doi.org/10.1007/BF02187705}
  {\path{doi:10.1007/BF02187705}}.

\bibitem{thomassen1986interval}
Carsten Thomassen.
\newblock Interval representations of planar graphs.
\newblock {\em Journal of Combinatorial Theory, Series B}, 40(1):9--20, 1986.
\newblock \href {https://doi.org/10.1016/0095-8956(86)90061-4}
  {\path{doi:10.1016/0095-8956(86)90061-4}}.

\bibitem{ueckerdt-phd}
Torsten Ueckerdt.
\newblock {\em Geometric Representations of Graphs with Low Polygonal
  Complexity}.
\newblock Doctoral thesis, Technische Universit\"at Berlin, Fakult\"at II -
  Mathematik und Naturwissenschaften, Berlin, 2012.
\newblock \href {https://doi.org/10.14279/depositonce-3190}
  {\path{doi:10.14279/depositonce-3190}}.

\bibitem{ung-53}
Peter Ungar.
\newblock On diagrams representing maps.
\newblock {\em Journal of the London Mathematical Society}, 28(3):336--342,
  1953.
\newblock \href {https://doi.org/10.1112/jlms/s1-28.3.336}
  {\path{doi:10.1112/jlms/s1-28.3.336}}.

\end{thebibliography}

\end{document}